\documentclass[a4paper, twoside]{scrartcl}
\pdfoutput=1
\usepackage{amsmath}
\usepackage{amssymb}
\usepackage{amsthm}
\usepackage{appendix}
\usepackage{bm}
\usepackage{centernot}
\usepackage{color}
\usepackage{enumitem}
\usepackage{fullpage}
\usepackage{hyperref}
\usepackage{mathtools}
\usepackage{natbib}
\usepackage{subcaption}
\usepackage{graphicx}
\usepackage{epstopdf}
\usepackage[capitalize,nameinlink]{cleveref} 

\usepackage{acro}
\acsetup{long-format=\normalfont\itshape, short-format=\normalfont\scshape}
\DeclareAcronym{uq}{short=uq, long=uncertainty quantification}
\DeclareAcronym{pdf}{short=pdf, long=probability density function}
\DeclareAcronym{cdf}{short=cdf, long=cumulative distribution function}
\DeclareAcronym{mcmc}{short=mcmc, long=Markov chain Monte Carlo}
\DeclareAcronym{iid}{short=iid, long=independent and identically distributed}
\DeclareAcronym{ks}{short=ks, long=Kolmogorov--Smirnov}
\DeclareAcronym{abc}{short=abc, long=approximate Bayesian computation}
\DeclareAcronym{cif}{short=cif, long=calibration influence function}
\DeclareAcronym{mks}{short=mks, long=multi-dimensional Kolmogorov--Smirnov}
\DeclareAcronym{nci}{short=nci, long=non-credibility index}
\DeclareAcronym{anees}{short=anees, long=average normalised estimation error squared}
\DeclareAcronym{pit}{short=pit, long=predictive integral transform}
\DeclareAcronym{gp}{short=gp, long=Gaussian process}
\DeclareAcronym{ode}{short=ode, long=ordinary differential equation}
\DeclareAcronym{pnm}{short=pnm, long=probabilistic numerical method}

\newcommand*{\rd}{\mathrm{d}}
\newcommand*{\wrt}{\rd}

\newcommand*{\defeq}{\coloneqq}
\newcommand*{\naturals}{\mathbb{N}}
\newcommand*{\reals}{\mathbb{R}}
\newcommand*{\tr}{^{\mkern-1.5mu\mathsf{T}}}

\newcommand{\norm}[1]{\lVert #1 \rVert}

\newcommand{\Absval}[1]{\left\vert #1 \right\vert}

\newcommand*{\arXiv}[1]{\bgroup\color{blue}\href{https://arxiv.org/abs/#1}{arXiv:#1}\egroup}

\usepackage{etex}
\usepackage[a4paper, top=88pt, bottom=88pt, left=72pt, right=72pt, headsep=16pt, footskip=28pt]{geometry}
\hypersetup{
	colorlinks,
	linkcolor=blue,
	citecolor=blue,
	urlcolor=blue,
}
\newcommand*{\doi}[1]{\bgroup\color{blue}\href{https://doi.org/#1}{doi:#1}\egroup}
\newcommand*{\email}[1]{\bgroup\color{blue}\href{mailto:#1}{#1}\egroup}
\renewcommand*{\url}[1]{\bgroup\color{blue}\href{#1}{#1}\egroup}
\usepackage{enumitem, moreenum}
\setlist[enumerate]{nosep}
\setlist[itemize]{nosep}
\usepackage{mleftright} \mleftright
\renewcommand{\qedsymbol}{$\blacksquare$}
\renewenvironment{proof}[1][\proofname]{\noindent{\bfseries\sffamily #1.} }{\hfill\qedsymbol\medskip}
\usepackage[labelfont={sf,bf}]{caption}
\DeclareCaptionLabelSeparator{figlabelsep}{\,\,\,}
\usepackage{scrlayer-scrpage, xhfill}
\automark[section]{section}
\setkomafont{pageheadfoot}{\normalcolor\sffamily}
\setkomafont{pagenumber}{\normalfont\normalsize\sffamily}
\clearpairofpagestyles
\let\oldtitle\title
\renewcommand{\title}[1]{\oldtitle{#1}\newcommand{\theshorttitle}{#1}}
\newcommand{\shorttitle}[1]{\renewcommand{\theshorttitle}{#1}}
\let\oldauthor\author
\renewcommand{\author}[1]{\oldauthor{#1}\newcommand{\theshortauthor}{#1}}
\newcommand{\shortauthor}[1]{\renewcommand{\theshortauthor}{#1}}
\cohead{\xrfill[0.525ex]{0.6pt}~\theshorttitle~\xrfill[0.525ex]{0.6pt}}
\cehead{\xrfill[0.525ex]{0.6pt}~\theshortauthor~\xrfill[0.525ex]{0.6pt}}
\newcommand{\theabstract}[1]{\par\bgroup\noindent\textbf{\textsf{Abstract.}} #1\egroup}
\newcommand{\thekeywords}[1]{\par\smallskip\bgroup\noindent\textbf{\textsf{Keywords.}}\newcommand{\and}{ $\bullet$ } #1\egroup}
\newcommand{\themsc}[1]{\par\smallskip\bgroup\noindent\textbf{\textsf{2020 Mathematics Subject Classification.}}\newcommand{\and}{ $\bullet$ } #1\egroup}
\newcommand*{\affilref}[1]{\ref{affiliation#1}}
\newcommand*{\affiliation}[3]{
	\footnotetext[#1]{\label{affiliation#2} #3}
}

\newcommand{\converteps}[1]{#1-eps-converted-to.pdf}

\numberwithin{equation}{section}
\numberwithin{figure}{section} 
\numberwithin{table}{section}

\newtheorem{theorem}{\sffamily Theorem}[section]
\newtheorem{corollary}[theorem]{\sffamily Corollary}

\newtheorem{lemma}[theorem]{\sffamily Lemma}
\theoremstyle{definition}
\newtheorem{definition}[theorem]{\sffamily Definition}
\newtheorem{remark}[theorem]{\sffamily Remark}
\newtheorem{example}[theorem]{\sffamily Example}

\crefname{conjecture}{Conjecture}{Conjectures}
\crefname{figure}{Figure}{Figures}
\crefname{question}{Question}{Questions}

\title{Testing whether a Learning \\ Procedure is Calibrated}
\shorttitle{Testing whether a Learning Procedure is Calibrated}

\author{%
	Jon~Cockayne\textsuperscript{\affilref{Soton}}%
	\and
	Matthew M.\ Graham\textsuperscript{\affilref{UCL}}%
	\and
	Chris J.\ Oates\textsuperscript{\affilref{Newcastle},\affilref{Turing}}%
	\and
	T.~J.\ Sullivan\textsuperscript{\affilref{Warwick},\affilref{Turing}}%
	\and
	Onur~Teymur\textsuperscript{\affilref{Kent}}
}
\shortauthor{J.\ Cockayne, M.~M.\ Graham, C.~J.\ Oates, T.~J.\ Sullivan, and O.\ Teymur}

\begin{document}

\maketitle
\affiliation{1}{Soton}{Mathematical Sciences, University of Southampton, Highfield, Southampton, SO17 1BJ, UK (\mbox{\email{jon.cockayne@soton.ac.uk}})}
\affiliation{2}{UCL}{Centre for Advanced Research Computing, University College London, Gower Street, London, WC1E 6BT, UK (\mbox{\email{m.graham@ucl.ac.uk}})}
\affiliation{3}{Newcastle}{School of Mathematics, Statistics \& Physics, Newcastle University, Newcastle upon Tyne, NE1 7RU, UK (\email{chris.oates@ncl.ac.uk})}
\affiliation{4}{Turing}{Alan Turing Institute, British Library, 96 Euston Road, London NW1 2DB, UK}
\affiliation{5}{Warwick}{Mathematics Institute and School of Engineering, University of Warwick, Coventry, CV4 7AL, UK (\email{t.j.sullivan@warwick.ac.uk})}
\affiliation{6}{Kent}{School of Mathematics, Statistics \& Actuarial Science, University of Kent, Cantebury, CT2 7NZ, UK (\email{o@teymur.uk})}

\begin{abstract}
\theabstract{A learning procedure takes as input a dataset and performs inference for the parameters $\theta$ of a model that is assumed to have given rise to the dataset.
Here we consider learning procedures whose output is a probability distribution, representing uncertainty about $\theta$ after seeing the dataset.
Bayesian inference is a prime example of such a procedure, but one can also construct other learning procedures that return distributional output.
This paper studies conditions for a learning procedure to be considered \textit{calibrated}, in the sense that the true data-generating parameters are plausible as samples from its distributional output.
A learning procedure whose inferences and predictions are systematically over- or under-confident will fail to be calibrated.
On the other hand, a learning procedure that \textit{is} calibrated need not be statistically efficient.
A hypothesis-testing framework is developed in order to assess, using simulation, whether a learning procedure is calibrated.
Several vignettes are presented to illustrate different aspects of the framework.}
\thekeywords{%
calibration%
\and%
credible sets%
\and%
uncertainty quantification%
}
\themsc{%
62A01
\and%
62F25
\and%
62F35
\and%
60J20
}
\end{abstract}

\section{Introduction}  

Given a parametric model and a dataset purported to be generated from the model, the modern workflow for parameter inference first identifies a statistical paradigm (e.g.\ Bayesian inference), performs any required numerical computation using an appropriate numerical method, then inspects the results and refines the approach until some \textit{desiderata} (e.g.\  posterior predictive checks, or a convergence diagnostic for a Markov chain Monte Carlo method) are satisfied.
This paper takes a holistic perspective and refers to the overall workflow as a \textit{learning procedure}. 
Our focus is on learning procedures that produce distributional output, examples of which include workflows based on Bayesian and generalised Bayesian inference \citep{Bissiri2016}, fractional posteriors \citep{Bhattacharya2019}, empirical Bayes \citep{Casella1985}, variational Bayes \citep{Blei2017}, approximate Bayesian computation \citep{beaumont2002approximate}, Bayesian synthetic likelihood \citep{price2018bayesian}, and also approaches that have a non-Bayesian motivation, such as the maximum entropy approach \citep{jaynes1982rationale}.

It is natural to hope that a learning procedure is \textit{calibrated}, in the sense that the true data-generating parameters are plausible as samples from the distributional output.
Indeed, a learning procedure that is \textit{not} calibrated can produce inferences and predictions that are either biased or over/under-confident, and lead users to draw spurious conclusions in model selection problems.
The consequences of over-confidence, in particular, could be dire when those inferences are used in safety-critical applications.
This point has been discussed at length in the literature, such as in investigating frequentist coverage of credible sets in Bayesian inference and in calibrating probabilistic forecasts.
However, the literature appears to lack a definition of ``calibration'' that is sufficiently general to be applied to an arbitrary learning procedure that produces distributional output.
The aim of this paper is to introduce a general definition of ``calibration'' and accompany this with a methodology for testing whether a learning procedure is calibrated.

The term \textit{calibration} is unfortunately overloaded in the statistical literature.
It is also used to refer to the parameter inference task in applications that involve a computer model. For example,  \cite{kennedy2001bayesian} write that ``\textit{the process of fitting the model to the observed data by adjusting the parameters is known as calibration}''. For avoidance of doubt, we use the standard terminology of \textit{parameter inference} to refer to the task of estimating parameters of a model.
The term `calibration' is also used in the literature on forecast assessment.
There the useage is close to the notions proposed in this paper, though in that literature the focus is on testing calibration at the level of the \emph{data} rather than at the level of the parameters.
This is discussed further in \cref{subsubsec: strong existing,subsubsec: related weak}.
We reserve the term \textit{calibration} for the specific notions proposed in this paper.

The outline of the paper is as follows:
\Cref{sec: methods} presents our proposed definitions, where we identify both \textit{strong} and \textit{weak} senses in which a learning procedure can be said to be calibrated.
To ensure our definitions are precise in a mathematical sense, we conceptualise a learning procedure as a mathematical object in \Cref{subsec: notation} and impose mild regularity assumptions on this object in \Cref{subsec: charac}.
In \Cref{subsec: strong calib} our notion of strong calibration is presented, illustrated by examples in \Cref{subsubsec: examples strong}, and compared to existing definitions in the literature in \Cref{subsubsec: strong existing}.
Likewise, in \Cref{subsec: weak calib} our notion of weak calibration is presented, illustrated by examples in \Cref{subsubsec: examples weak}, and compared to existing definitions in the literature in \Cref{subsubsec: related weak}.
Several vignettes are provided in \Cref{sec: simulations}, showing through simulations that our proposed definitions of calibration both accord with intuition and can be tested for.
A brief discussion concludes the paper in \Cref{sec: discuss}.

\subsection{Notation}

For a measurable space $S$, $\mathcal{P}(S)$ will denote the set of probability measures on $S$.
For $s \in S$ let $\delta(S) \in \mathcal{P}(s)$ denote the Dirac distribution on $s$.
For a measurable function $f \colon S \to \reals$, a measurable set $A \subseteq \reals$, and a probability measure $\nu \in \mathcal{P}(S)$, let $f^{-1}(A) \defeq \{x \in S \mid f(x) \in A\}$ denote the \emph{preimage} of $A$ and recall that the \emph{pushforward} measure $f_\# \nu \in \mathcal{P}(\reals)$ is defined as $(f_\#\nu)(A) \defeq \nu(f^{-1}(A))$.

\section{What it Means for a Learning Procedure to be Calibrated} \label{sec: methods}

This section sets out our proposed definitions of strong and weak calibration, provides examples of learning procedures that are strongly and weakly calibrated, and relates our definitions to existing work.

\subsection{Set-Up} \label{subsec: notation}


Let $\Theta$ be a measurable space, which will play the role of the \textit{parameter space} in this work.
It is assumed that there is a unique ``true'' parameter $\theta \in \Theta$ and we consider the \textit{parameter inference} task of estimating $\theta$ based on a dataset.
Let $Y$ be a measurable space in which datasets are realised.

\begin{definition}[Learning Procedure] \label{def: learning proc}
A \emph{learning procedure} is a function
\begin{align*}
\mu \colon \mathcal{P}(\Theta) \times Y & \to \mathcal{P}(\Theta) \\
(\mu_0,y) & \mapsto \mu(\mu_0,y) .
\end{align*}
Here $\mu_0$ is interpreted as an initial \emph{belief distribution}, quantifying uncertainty about the parameter $\theta$ before any data have been observed, and $y$ denotes a dataset.
The distributional output $\mu(\mu_0,y)$ is interpreted as a quantification of the uncertainty associated with the parameter $\theta$, after the data $y$ have been observed.
\end{definition}

The standard example of a learning procedure is Bayesian inference, wherein $\mu_{0}$ is the prior distribution and $\mu (\mu_{0}, y)$ is the posterior distribution, this being determined by the prior, the observed data $y$, and a likelihood function that must be specified.
However, \Cref{def: learning proc} is general enough to accommodate any workflow that produces distributional output.
In particular, \Cref{def: learning proc} does not pre-suppose that a data-generating model exists or is known to the user, so that the definition of a learning procedure may be applied even in the \emph{M-open} setting \citep[\S6.1.2]{Bernardo2000}.
Further, one may consider that computational procedures such as variational inference or Monte Carlo form part of the learning procedure, and in this sense a myriad of different learning procedures can be considered.

Note that we call $\mu_0$ a \textit{belief distribution} following \cite{Bissiri2016} and reserve the term \textit{prior} for use only in the Bayesian context.
We also emphasise that a learning procedure need not depend upon the initial belief distribution $\mu_0$; for example, in the maximum entropy approach \citep{jaynes1982rationale} a distributional output is produced that does not explicitly depend on any initial belief, so that effectively $\mu(\mu_0, y) \equiv \mu(y)$.

In the next section we will introduce the mathematical facts required for our notions of strong and weak calibration in \Cref{subsec: strong calib,subsec: weak calib}.

\subsection{A Mathematical Characterisation} \label{subsec: charac}

The definitions that we will present rely on \acp{cdf} and their inverses, and we therefore impose regularity conditions to ensure that such inverse \acp{cdf} are well-defined.
That is, we impose sufficient regularity to restrict our attention in the sequel to inverse \acp{cdf} that are well-defined functions, as opposed to dealing with generalised functions that are set-valued.

\begin{definition}[Regular Distribution] \label{def: reg def}
Let $\Theta$ be a measurable space equipped with a reference measure $\lambda$.
A distribution $\nu \in \mathcal{P}(\Theta)$ is \emph{regular} (with respect to $\lambda$) if it admits a \ac{pdf} $p_\nu \defeq \rd \nu / \rd \lambda$ such that $p_\nu > 0$ on $\Theta$ (i.e. the measures $\nu$ and $\lambda$ are \emph{equivalent}).
The set of all regular distributions will be denoted $\mathcal{P}_r(\Theta)$.
\end{definition}

\noindent When $\Theta$ is a Borel- or Lebesgue-measurable subset of Euclidean space, the reference measure $\lambda$ will be assumed to be Lebesgue measure.
For $-\infty \leq a < b \leq \infty$ and a univariate distribution $\gamma \in \mathcal{P}((a,b))$, we let $F_\gamma \colon (a,b) \to [0,1]$ denote the associated \ac{cdf} $F_\gamma(x) \defeq \gamma((a,x]) = \int_a^x \rd \gamma$.
Our first result, \Cref{lem: reg to uniform}, is classical \citep[e.g.][]{rosenblatt1952remarks} and underpins methods for simulation of univariate random variables using inverse \acp{cdf}.
This result establishes that the level of regularity in \Cref{def: reg def} is sufficient for the inverse \ac{cdf} approach to simulation of such distributions to be applied.
It also ensures that our subsequent constructions that depend on \Cref{def: reg def} are well-defined.

\begin{lemma} \label{lem: reg to uniform}
For $-\infty \leq a < b \leq \infty$ and $\nu \in \mathcal{P}_r((a,b))$ we have that $F_\nu(X) \sim \mathcal{U}(0,1)$ whenever $X \sim \nu$.
\end{lemma}
\begin{proof}
Since $\nu$ admits a \ac{pdf} $p_\nu$ on $(a,b)$, the fundamental theorem of calculus implies that $F_\nu$ is differentiable with $DF_\nu(\theta) = p_\nu(\theta)$.
In particular, since $p_\nu > 0$ we have that $F_\nu$ is continuous and strictly increasing and therefore the sets $F_\nu^{-1}(z)$ are singletons for all $z \in [0,1]$.
Let $X \sim \nu$ and $Z \defeq F_\nu(X)$.
Then, from the change of variables formula, $Z$ admits a \ac{pdf} $q(z)$ on $[0,1]$ with
\[
q(z) = \sum_{\theta : F_\nu(\theta) = z} p_\nu(\theta) |DF_\nu(\theta)|^{-1} = p_\nu(\theta) \frac{1}{p_\nu(\theta)} = 1 ,
\]
which is indeed the \ac{pdf} of $\mathcal{U}(0,1)$.
\end{proof}

\noindent The random variable $F_\nu(X)$ is sometimes called the \textit{probability integral transform}; see e.g. \cite{Dawid1984,diebold1997evaluating}.
When $\Theta\not\subset\reals$, the \ac{cdf} of a distribution $\nu \in \mathcal{P}_r(\Theta)$ is not in general well-defined.
To characterise such distributions analogously to the above, consider a set of test functions of the form $f \colon \Theta \to (a,b)$, with the property that each univariate marginal $f_\# \nu$ \emph{does} admit an invertible \ac{cdf}.
We next establish that regular distributions are characterised by a certain (large) set of such statistics.


\begin{definition}[Test Functions $\mathcal{F}_\Theta$]
Consider measurable functions of the form $f \colon \Theta \to (a,b)$ for some $-\infty \leq a < b \leq \infty$ .
Then the \emph{test functions} $\mathcal{F}_\Theta$ are the set of all such $f$ for which $f_\#\nu \in \mathcal{P}_r((a,b))$ whenever $\nu \in \mathcal{P}_r(\Theta)$.
\end{definition}

\noindent Intuitively, $\mathcal{F}_\Theta$ rules out functions $f$ that take a constant value on a non-null set, in order to avoid the situation where $f_\# \nu$ contains an atom and the \ac{cdf} $F_{f_\# \nu} \colon (a,b) \to [0,1]$ is not invertible.
In the univariate case $\Theta = \reals$, the set $\mathcal{F}_\Theta$ contains functions $f$ for which the gradient exists and is nonzero almost everywhere and, moreover, the preimages $f^{-1}(z)$ have cardinality $n$ such that $0 < n < \infty$ for each $z \in (a,b)$.
Indeed, in this case $f_\# \nu$ admits an everywhere positive (Lebesgue) \ac{pdf} on $(a,b)$ of the form
\begin{equation}
p_{f_\# \nu}(z) = \sum_{\theta \in f^{-1}(z)} p_\nu(\theta) \left|Df(\theta) \right|^{-1} . \label{eq: pdf d1}
\end{equation}
Since by assumption $f_\# \nu$ is regular on $(a,b)$, from \Cref{lem: reg to uniform} we have that $F_{f_\# \nu}(f(\theta)) \sim \mathcal{U}(0,1)$ whenever $X \sim \nu$.
For the multivariate case $\Theta = \reals^d$, by the co-area formula the (Lebesgue) \ac{pdf} of $f_\#\nu$ is
\begin{equation}
p_{f_\# \nu}(z) = \int_{f^{-1}(z)} p_\nu(\theta) |\text{det}(Df(\theta)Df(\theta)\tr)|^{-\frac{1}{2}} \,\mathcal{H}^{d-1}(\rd \theta) , \label{eq: pdf dd}
\end{equation}
where $\mathcal{H}^{d-1}$ indicates the $(d-1)$ dimensional Hausdorff measure on $\Theta$ \citep[Proposition 2]{diaconis2013sampling}.
In this case, the requirement on the Jacobian determinant is that $\text{det}(Df Df\tr) \neq 0$ almost everywhere.
As $\mathcal{H}^{0}$ is equivalent to the counting measure, \eqref{eq: pdf dd} collapses back to \eqref{eq: pdf d1} when $d = 1$.

The restriction of attention to $\mathcal{F}_\Theta$ is essentially without loss of generality, as evidenced by the following result, whose proof is contained in \Cref{sec:proof:weak_control}:

\begin{lemma}[Regular Distributions are Characterised by $\mathcal{F}_\Theta$] \label{lem:weak_control}
Let $\Theta = \reals^d$ for some $d \in \naturals$.
Suppose that $\mu,\nu \in \mathcal{P}_r(\Theta)$ and $\int f \, \rd \mu = \int f \, \rd \nu$ for all $f \in \mathcal{F}_\Theta$.
Then $\mu = \nu$.
\end{lemma}

Now we have the mathematical tools to define what it means for a learning procedure to be calibrated.
In \Cref{subsec: strong calib} we introduce a strong notion of calibration, which clarifies the sense in which the true parameter can be considered plausible as a sample from the distributional output.
Then, in \Cref{subsec: weak calib}, we consider a strictly weaker notion of calibration that is more easily tested.

\subsection{Strongly Calibrated Learning Procedures} \label{subsec: strong calib}



To assess whether a learning procedure is calibrated we must specify what it is calibrated \textit{against}, and this requires a data-generating model.
Thus, the assessment framework we present exists in the \emph{M-complete} setting \citep[\S6.1.2]{Bernardo2000}.

\begin{definition}[Data-Generating Model]
A \emph{data-generating model} is a function
\begin{align*}
P \colon \Theta & \to \mathcal{P}(Y) \\
\theta & \mapsto P_\theta ,
\end{align*}
where $P_\theta$ carries the interpretation of a statistical model from which data are generated.
\end{definition}

In this section we present a strong notion of what it means for a learning procedure to be calibrated to a data-generating model.
It simplifies matters to restrict to learning procedures that produce regular distributional output:

\begin{definition}[Regular Learning Procedure] \label{def:reg_learning_proc}
A learning procedure $\mu \colon \mathcal{P}(\Theta) \times Y \to \mathcal{P}(\Theta)$ is \emph{regular} if $\mu(\mu_0,y) \in \mathcal{P}_r(\Theta)$ for all $\mu_0 \in \mathcal{P}_r(\Theta)$ and all $y \in Y$.
\end{definition}

\begin{definition}[Strongly Calibrated] \label{def: strong}
Let $B \subseteq \mathcal{P}_r(\Theta)$ denote a set of belief distributions and $P$ a data-generating model.
A regular learning procedure $\mu$ is said to be \emph{strongly calibrated} to $(B, P)$ if
\begin{align}
\left.
\begin{array}{rl}
\theta & \sim \mu_0 \\
y \mid \theta & \sim P_\theta
\end{array}
\right\} \implies F_{f_\# \mu(\mu_0,y)}(f(\theta)) \sim \mathcal{U}(0,1)  \label{eq: strong calib}
\end{align}
for all $f \in \mathcal{F}_\Theta$ and for all $\mu_0 \in B$.
If the set $B$ contains a single element, $\mu_0$, then we say simply that $\mu$ is strongly calibrated to $(\mu_0,P)$.
\end{definition}

The assumption that both the belief distribution and learning procedure are regular excludes some important learning procedures.
For example, in Bayesian inference one sometimes uses an improper, ``uninformative'' prior such as $p(\theta) \propto 1$, which would not be regular unless $\Theta_0$ is bounded.
To study such a learning procedure in the framework of \cref{def: strong} one could consider constructing an ``artificial'' learning procedure that took a regular distribution $\mu_0$ as input, but ignored this for the purposes of inference and instead used an improper prior---though, one would still need to ensure that the learning procedure itself returned a regular output, which is not guaranteed for an improper prior.
In addition to this, any application of Bayesian inference for which the support of the posterior is a strict subset of $\Theta$ (e.g.\ procedures with truncated likelihoods) will fail to be regular.
The distributional output of \ac{abc} may not be regular for similar reasons.
This motivates the introduction of \emph{weakly calibrated} learning procedures in \cref{subsec: weak calib}, for which the regularity assumption can be relaxed.

To gain intuition for \cref{def: strong}, notice that the unknown data-generating parameter $\theta$ is statistically identical to a sample from the distributional output $\mu(\mu_0,y)$ when the learning procedure is strongly calibrated.
This intuition is clarified in the following remark:

\begin{remark}[Correct Coverage for Credible Sets]
Suppose that the learning procedure $\mu$ is strongly calibrated to $(\mu_0,P)$.
If the distribution $\mu(\mu_0,y)$ is used to construct a $1-\alpha$ probability credible set for $\theta$, then this interval will indeed contain $\theta$ with probability $1-\alpha$ under the hierarchical data-generating model $\theta \sim \mu_0$, $y \mid \theta \sim P_\theta$.
\end{remark}

\noindent Thus, the distributional output from a strongly calibrated learning procedure can be meaningfully related to the parameter inference task.
Note, however, that even a small degree of misspecification can lead to failure of calibration.
Thus strong calibration captures the absence of systematic errors, similar to the notion of an unbiased estimator. 

Next we present an actionable test for the hypothesis that a learning procedure is strongly calibrated.
We emphasise that this test can in theory be applied to \emph{any} learning procedure (i.e.\ any workflow used for parameter inference that returns distributional output), providing that the regularity requirements are satisfied and that one is able to simulate from the data-generating model.

\begin{remark}[Testing whether a Learning Procedure is Strongly Calibrated] \label{rem: GOF}
Fix $\mu_0 \in \mathcal{P}_r(\Theta)$ and let
\begin{align*}
\theta_i & \stackrel{\text{\ac{iid}}}{\sim} \mu_0 \\
y_i \mid \theta_i & \stackrel{\text{\ac{iid}}}{\sim} P_{\theta_i}
\end{align*}
Then we can test whether a (regular) learning procedure $\mu$ is strongly calibrated to $(\mu_0,P)$ by picking a test function $f \in \mathcal{F}_\Theta$ and using any goodness-of-fit test for the hypothesis
\begin{align*}
F_{f_\# \mu(\mu_0,y_i)}(f(\theta_i)) \stackrel{\text{\ac{iid}}}{\sim} \mathcal{U}(0,1) .
\end{align*}
Such a test will not have power against all alternatives unless, for example, $d=1$ and $f(\theta) = \theta$.
To increase the power of the test in higher dimensions, multiple $f$ should be simultaneously considered.
Methodology for selecting a suitable test function is proposed in \Cref{subsec: data-driven}. 
\end{remark}

\begin{remark}
For simplicity we have assumed that each $\theta_i$ is associated with exactly one $y_i$.
In practice this need not be the case; each parameter could be associated with many pieces of data.
For example in some applications a sample from $\mu_0$ may be more difficult to obtain than repeated measurements $y_1, \dots, y_n \sim P_{\theta_i}$.
However we note that this will violate the independence assumption in \cref{rem: GOF}, and would require a more complicated test to be used.
\end{remark}

\begin{remark}[Quantification of Strong Calibration] \label{rem: quant strong}
The departure from uniformity of the law of $F_{f_\# \mu(\mu_0,y)}(f(\theta))$ under $\theta \sim \mu_0$, $y \mid \theta \sim P_\theta$ can be used to assess the nature and extent to which the learning procedure fails to be strongly calibrated.
Histograms can provide an intuitive visualisation; see \Cref{subsec: ODE calib}.
\end{remark}

In the next section we illustrate \Cref{def: strong} with some examples for which strong calibration can be verified.
Then, in \Cref{subsubsec: strong existing} we discuss the relationship between \Cref{def: strong} and earlier work.

\subsubsection{Examples of Strongly Calibrated Learning Procedures}
\label{subsubsec: examples strong}

Our first example confirms the intuition that the Bayesian framework is strongly calibrated to the prior and the data-generating model.

\begin{example}[Bayes is Strongly Calibrated] \label{ex: Bayes strong}
If $\theta \sim \mu_0$ and $y \mid \theta \sim P_\theta$ then $(\theta,y)$ can be considered to be a sample from the joint distribution of the parameters and dataset.
In the Bayesian framework (with the data-generating model $P$ correctly specified), $\mu(\mu_0,y)$ is defined as the conditional distribution of the parameters given the data, and thus $\theta \mid y \sim \mu(\mu_0,y)$.
Thus if $\mu_0$ and $\mu$ are regular, it follows from \Cref{lem: reg to uniform} that $F_{f_\# \mu(\mu_0,y)}(f(\theta)) \sim \mathcal{U}(0,1)$ for all $f \in \mathcal{F}_\Theta$.
Thus Bayesian inference is strongly calibrated to $(\mathcal{P}_r(\Theta), P)$.
\end{example}

The following example\footnote{This example is similar in spirit to the \emph{climatological forecaster} in Example 2 of \citet{gneiting2007probabilistic}, who uses only historical frequencies to predict tomorrow's weather, agnostic of any recent data that may have been obtained.} shows that strongly calibrated learning procedures do not necessarily yield accurate estimators:

\begin{example}[Data-Agnostic Learning Procedure is Strongly Calibrated] \label{eg:trivial strong}
The trivial learning procedure that takes $\mu(\mu_0,y) \defeq \mu_0$ is strongly calibrated to $(\mathcal{P}_r(\Theta), P)$, since for $\theta \sim \mu_0$ and $\mu_0 \in \mathcal{P}_r(\Theta)$,
\[
F_{f_\# \mu(\mu_0,y)}(f(\theta)) = F_{f_\# \mu_0}(f(\theta)) \sim \mathcal{U}(0,1)
\]
for all $f \in \mathcal{F}_\Theta$.
\end{example}

\noindent The implication of \Cref{eg:trivial strong} is that strong calibration alone is not sufficient to justify the practical application of a learning procedure, and additional \textit{desiderata}, such as \textit{statistical efficiency}, will typically also need to be taken into account.\footnote{For example, ``\textit{maximizing the sharpness of
the predictive distributions subject to calibration}'' was proposed in \citet{gneiting2007probabilistic}, although their use of the term ``calibration'' is distinct from the present paper, being focussed on forecast assessment. See \Cref{subsubsec: strong existing} for further discussion of the literature on forecast assessment.}
This paper focusses on calibration and does not attempt to discuss other \textit{desiderata} and how they should be balanced in the applied context.

One can consider situations between the two extremes of \Cref{ex: Bayes strong} and \Cref{eg:trivial strong}:

\begin{example}[Partial Posteriors are Strongly Calibrated] \label{ex: partial strong}
A \emph{partial posterior} corresponds to performing full Bayesian inference using only summary statistics $s \colon Y \to S$ of the dataset.
These have recently been proposed as a tool for compensating for model misspecification \citep{Lewis2021}.
For the partial posterior learning procedure, $\mu(\mu_0,y)$ is the conditional distribution of the parameters given the summarised data $s(y)$ and $\theta \mid s(y) \sim \mu(\mu_0,y)$.
When both the prior and the partial posterior learning procedure are regular, it follows from \Cref{lem: reg to uniform} that $F_{f_\# \mu(\mu_0,y)}(f(\theta)) \sim \mathcal{U}(0,1)$ for all $f \in \mathcal{F}_\Theta$.
Thus partial posteriors are strongly calibrated to $(\mathcal{P}_r(\Theta), P)$.
\end{example}

Next we present an example that is a clear departure from the Bayesian framework, in that it clearly does not return a posterior distribution and yet is provably strongly calibrated:

\begin{example}[Probabilistic Stationary Iterative Methods are Strongly Calibrated] \label{ex: iterative methods}
Let $\Theta = \reals^d$ and consider the data-generating model $P_\theta = \delta(A \theta)$ that returns a Dirac distribution on $y = A \theta$, where $A$ is a non-singular matrix.
An ideal learning procedure would return $\mu(\mu_0,y) = \delta(A^{-1} y) = \delta(\theta)$, but in many practical scenarios the exact action of $A^{-1}$ on $y$ cannot be computed, either due to poor conditioning of the matrix $A$ or due to the $O(d^3)$ computational cost associated with inverting $A$.
This motivates the use of an alternative procedure, called a \emph{probabilistic iterative method}, recently proposed in \cite{Cockayne2020} and based on classical iterative methods for solving linear systems \citep[see e.g.~][]{Saad2003}.
To describe the procedure, let $R^y: \Theta \to \Theta$ be an map, constructed using $y$, such that $\theta$ is a solution of the fixed point equation $\theta = R^y(\theta)$.
For example, the choice $R^y(\theta) = (I - \epsilon A) \theta + \epsilon y$, $\epsilon > 0$, corresponds to a classical iterative method called \textit{Richardson's method}.
Consider then the learning procedure $\mu(\mu_0,y) \defeq R^y_\# \mu_0$, whose output is conjugate under a Gaussian input $\mu_0$, being an affine transform, and can be exactly computed at cost $O(d^2)$.
\cite{Cockayne2020} proved that, under mild conditions, the iterative application of $R^y$ produces a sequence of distributions on $\Theta$ that contract to $\delta(\theta)$, and that this procedure is strongly calibrated to $(G(\reals^d), P)$, where $G$ is the set of all Gaussian distributions supported on $\reals^d$.
This example speaks to one potential use of \Cref{def: strong}, in providing theoretical justification for non-traditional learning procedures which nevertheless produce meaningful distributional output.
\end{example}

Next our attention turns to the relationship between \Cref{def: strong} and existing concepts in the literature.

\subsubsection{Relation to Existing Concepts} \label{subsubsec: strong existing}

Here we compare and contrast our notion of strong calibration with concepts appearing in earlier work and in related fields.

\paragraph{Frequentist Coverage:}

There is a rich literature that aims to assess learning procedures according to frequentist \textit{desiderata}.
In particular, one can ask whether credible sets have \textit{correct frequentist coverage}, which is analogous to fixing $\theta = \theta_0$ and asking if $y \sim P_{\theta_0}$ implies $F_{f_\# \mu(\mu_0,y)}(f(\theta_0)) \sim \mathcal{U}(0,1)$; i.e.\ the only randomness is introduced during generation of the dataset.
This differs to our notion of strong calibration in that we sample $\theta$ from $\mu_0$ while, in the frequentist assessment, $\theta$ is fixed.
In particular, it is possible to prove certain learning procedures are strongly calibrated, but no learning procedure can be expected to attain correct frequentist coverage in general.
The literature on frequentist assessment therefore focuses on weaker notions of coverage, such as \textit{asymptotically correct frequentist coverage}, where the data are of the form $y = (y_1,\dots,y_n)$ and credible sets are required to have correct frequentist coverage in the $n \to \infty$ limit.
In finite-dimensional Bayesian analyses where a Bernstein--von--Mises theorem holds, asymptotically correct frequentist coverage is guaranteed \citep{freedman1999wald}.
Results on frequentist coverage have also been established in finite dimensions for variational Bayes \citep{Wang2019}.
In infinite-dimensional settings, a Bayesian learning procedure can fail to have even asymptotically correct frequentist coverage \citep{cox1993analysis,freedman1999wald}.
An active area of research is to establish sufficient conditions for asymptotically correct frequentist coverage, and recent results have been established that hold uniformly over a set of values for $\theta_0$; for results in this direction see \cite{Szabo2015} and references therein.

\paragraph{Forecast Assessment:}

\citet{Dawid1984} refers to the question of whether a probabilistic forecasting system is in some sense ``good'' as ``\textit{the fundamental question of prequential statistics}''.
Our notion of strong calibration is closely related to a concept developed in that literature to answer this question, for which the term \emph{probabilistic calibration} is used  \citep{dawid1982well,diebold1997evaluating,gneiting2007probabilistic,gneiting2013combining}.
An important distinction between forecast assessment and the present paper is the sense in which probabilistic calibration is applied; here we estimate a ``true'' parameter $\theta$, which is not a random variable, whereas in forecast assessment there remains inherent randomness in the quantities being predicted.

In the econometrics literature, \citet{diebold1997evaluating} considered a sequence of forecasts $(Q_i)_{i=1}^n \subset \mathcal{P}_r(\mathbb{R})$, representing predictions for corresponding quantities $(q_i)_{i=1}^n \subset \mathbb{R}$.
The authors advocated a visual diagnostic, called a \emph{correlogram}, to assess whether $\{F_{Q_i}(q_i)\}_{i=1}^n$ are plausible as an independent random sample from $\mathcal{U}(0,1)$; see also \cite{christoffersen1998evaluating,berkowitz2001testing}.
In the statistics community, \citet{gneiting2013combining} proposed to compare the variance of the $\{F_{Q_i}(q_i)\}_{i=1}^n$ to $1/12$, the variance of a $\mathcal{U}(0,1)$ random variable, with the sequence of forecasts being called \emph{overdispersed} if this variance is smaller  than 1/12, and \emph{underdispersed} if it is larger; see the review of \citet{gneiting2014probabilistic}.
This literature contains elements that are similar in spirit to our notion of strong calibration, except that a parametric statistical model is not explicitly involved; an important distinction that we require when assessing whether a learning procedure is calibrated.

In the meteorology literature, the calibration of probabilistic forecasts is routinely assessed using \emph{rank histograms} \citep{anderson1996method,talagrand1997evaluation,hamill1997verification,hamill2001interpretation}. For computational reasons, a forecast is typically represented by a discrete distribution $\mu(\mu_0,y) \approx \frac{1}{M} \sum_{m=1}^M \delta(\theta^m)$, produced based on initial belief $\mu_0$ and after observing data $y$, assumed to have arisen from a data-generating model $P$.
To assess the forecast, an ensemble of synthetic datasets $\lbrace y^{m} \rbrace_{m=1}^M$ is simulated as $y^m \sim P_{\theta^m}$.
For a test function $f \in \mathcal{F}_Y$, the rank statistic
\[
r(\lbrace f(y^{m}) \rbrace_{m=1}^M, f(y)) \defeq \sum_{m=1}^M \mathbb{I}[f(y^{m}) < f(y)]
\]
will be uniformly distributed on $\lbrace 0, 1, \dots, M \rbrace$ if the forecast is calibrated.
This is assessed empirically by producing a histogram of rank statistics for a collection of $T$ ensembles of synthetic datasets $\lbrace \lbrace y_t^m \rbrace_{m=1}^M\rbrace_{t=1}^T $ and corresponding real datasets $\lbrace y_t \rbrace_{t=1}^T$, where $t = 1, \dots, T$ may index distinct times, spatial locations, or both.
Denoting the empirical measure associated with an ensemble of synthetic datasets as $\hat{\nu}_t = \tfrac{1}{M}\sum_{m=1}^M\delta(y^m_t)$, the rank statistic $r(\lbrace f(y^m_t) \rbrace_{m=1}^M, f(y_t))$ is related to the \ac{cdf} of $\hat{\nu}_t$ by
\[
F_{f_\#\hat{\nu}_t}(f(y_t)) = \frac{1}{M}r(\lbrace f(y^m_t) \rbrace_{m=1}^M, f(y_t)).
\]
Checking for rank histogram uniformity is therefore similar in spirit to the test for strong calibration in \Cref{rem: GOF}, with relaxations to allow for the fact that the learning procedure produces an empirical distribution output and that the true parameters $\{\theta_t\}_{t=1}^T$ that gave rise to the real datasets $\{y_t\}_{t=1}^T$ are unknown, so that testing occurs in the data domain $Y$ rather than in the parameter domain $\Theta$.



\paragraph{Signal Processing:}

An important goal in signal processing is to estimate a time-dependent latent state $\lbrace\theta_t\rbrace_{t=1}^T$, $\theta_t \in \reals^d$, based on time-series data $\lbrace y_t \rbrace_{t=1}^T$.
For Gaussian filtering algorithms, such as the extended Kalman filter \citep[see][p84]{law2015data}, the output of the learning procedure is a sequence of Gaussian distributions $\mathcal{N}(m_t, \Sigma_t)$.
These serve to quantify uncertainty as to the unknown value of the parameter $\theta_t$, $t = 1,\dots,T$.
Such a filtering algorithm is considered to be calibrated if the \emph{Z-score} $\Sigma_t^{-1/2}(\theta_t - m_t)$ is plausible as a sample from $\mathcal{N}(0,1)$.
The \ac{anees} \citep{bar1983consistency,drummond1998comparison}
\[
\frac{1}{T} \sum_{t=1}^T (\theta_t - m_t)^\top \Sigma_t^{-1} (\theta_t - m_t)
\]
attempts to quantify this property, with values of \ac{anees} close to 1 when the learning procedure is calibrated.
\cite{li2002estimator} argued against the use of \ac{anees} on the grounds that it ``\textit{penalises optimism much more severely than pessimism}''.\footnote{It is unclear to us whether this is a problem, since in most statistical applications estimates that are conservative are generally preferred to estimates that are over-confident.}
These authors then proposed the \ac{nci}
\[
\frac{10}{T} \sum_{t=1}^T \log_{10}\left( \frac{(\theta_t - m_t)^\top \Sigma_t^{-1} (\theta_t - m_t)}{(\theta_t - m_t)^\top \bar{\Sigma}_t^{-1} (\theta_t - m_t)} \right)
\]
where $\bar{\Sigma}_t$ is the covariance matrix of the random vector $\theta_t - m_t$, where the randomness here refers to the generation of the dataset.
The \ac{nci}, which is also called the \emph{inclusion indicator} in \cite{Li2006}, takes values close to 0 if the filtering algorithm is calibrated and is quite widely used \citep[e.g.][]{pruher2020improved}.
Further discussion can be found in \cite{li2011evaluation}.
The \ac{anees} is similar in spirit to our \Cref{def: strong}, but it is adapted to learning procedures that produce Gaussian output and to a temporal data-generating model.

\paragraph{Validation of Algorithms for Bayesian Computation:}

\cite{Cook2006} observed that Bayesian inference is strongly calibrated to the prior and the data-generating model\footnote{Though, the result was not described in such terms in that work.} and presented the argument used in \Cref{ex: Bayes strong}.
Their interest was in validating software for Bayesian inference, and general learning procedures were not considered.
They proposed a goodness-of-fit test for the case $\Theta = \reals^d$ that corresponds to \Cref{rem: GOF}, using a test statistic of the form
\begin{equation}
T \defeq \sum_{i=1}^n (F_{\mathcal{N}(0,1)}^{-1}(F_{f_\#\mu(\mu_0,y_i)}(f(\theta_i))))^2 \label{eq: cook}
\end{equation}
for some $f \in \mathcal{F}_\Theta$.
If the null hypothesis holds and the learning procedure is strongly calibrated, then $T \sim \chi_n^2$.
\cite{Cook2006} focused on software that uses \ac{mcmc}, meaning that \ac{cdf}s are not exactly computed, and advocated an empirical approximation to the \ac{cdf} based on approximate samples $\lbrace \theta_i^m \rbrace_{m=1}^M$ from $\mu(\mu_0,y_i)$ generated using \ac{mcmc}.

A similar approach was used to analyse \ac{abc} in \cite{wegmann2009efficient}, who performed a \ac{ks} test for uniformity, and in \cite{prangle2014diagnostic} who used the name \textit{coverage property} and advocated a visual diagnostic plot.
In more recent work, \cite{lee2019calibration,xing2019calibrated} proposed the use of credible sets to circumvent access to \ac{cdf}s; this is similar in spirit to taking $f$ to be an indicator function in \Cref{def: strong}.
In \cite{Talts2018} the authors modified the approach of \cite{Cook2006} to address issues surrounding empirical approximation of the \ac{cdf}, such as discretisation artefacts when displayed as a histogram if an appropriate continuity correction or binning scheme is not used.
\citet{Talts2018} showed that, for \ac{iid} samples $\lbrace \theta_i^m \rbrace_{m=1}^M$ from the posterior given $y_i$, rank statistics $r(\lbrace f(\theta^m_i) \rbrace_{m=1}^M, f(\theta_i))$ for a test function $f \in \mathcal{F}_\Theta$ will follow a discrete uniform distribution on $\lbrace 0,1, \dots M \rbrace$, and proposed to use this to test calibration rather than checking the (continuous) uniformity of estimated quantiles. Further, \citet{Talts2018} proposed to alleviate departures from uniformity in the rank statistics arising from the use of dependent \ac{mcmc} rather than \ac{iid} samples by thinning the \ac{mcmc} samples using a heuristic based on the estimated chain autocorrelations.


\paragraph{Validation of Bayesian Workflows:}
The aforementioned authors including \cite{Cook2006} focussed on the correctness of algorithms for Bayesian computation, but one can take a broader view in which a \textit{Bayesian workflow} (e.g. including prior elicitation, selection of a likelihood, and so forth; see \cite{Gelman2020}), also form part of the learning procedure to be assessed.
The earliest related work in this direction of which we are aware is \cite{monahan1992proper}, who stated a definition similar to our strong calibration (albeit in terms of credible sets).
These authors considered generalised Bayesian inference and provided the argument used in \Cref{ex: partial strong}.
A \ac{ks} test for uniformity of $F_{\mu(\mu_0,y_i)}(\theta_i)$ was proposed in the case where $\Theta$ is one-dimensional.

\cite{harrison2015validation} proposed a notion of calibration that is similar in spirit to our \Cref{def: strong}, motivated by the often challenging computational workflows encountered in applications to astronomy.
First, the authors take a collection of candidate values $\theta_i$ for the parameter and generate associated datasets $y_i \mid \theta_i \stackrel{\text{\ac{iid}}}{\sim} P_{\theta_i}$.
The values $\theta_i$ ``\textit{may be the same for each simulation
generated or differ between them, depending on the nature of the inference problem}''.
Then, recasting into our notation, these authors proposed to ``\textit{test the null hypothesis that each set of assumed parameter values $\theta_i$ is drawn from the corresponding derived posterior $\mu(\mu_0,y)$}''.
This procedure coincides with our notion of strong calibration only if $\theta_i \stackrel{\text{\ac{iid}}}{\sim} \mu_0$.
The authors considered Bayesian workflows (``\textit{our validation procedure [...] allows for the verification of the implementation and any simplifying assumptions of the data model}'') and proposed a ``\textit{multiple simultaneous version of [a novel, multi-dimensional] Kolmogorov--Smirnov test}'' for the calibrated null hypothesis.
This \ac{mks} test provides an ingenious way to circumvent the selection of a test function $f$ in \Cref{rem: GOF}, being based on highest probability density regions instead of \ac{cdf}s.
However, the \ac{mks} test does not have power against all alternatives to the calibrated null hypothesis, even in dimension $d = 1$, and the description of the test as a multi-dimensional \ac{ks} test is misleading, as when $d=1$ the test does not correspond to a standard \ac{ks} test.

\paragraph{Summary:}

In summary, the content of \Cref{subsec: notation,subsec: charac,subsec: strong calib,subsubsec: examples strong} departs from existing work on this topic in that:
\begin{enumerate}
\item where similar hypothesis tests have been performed in \cite{monahan1992proper,Cook2006,harrison2015validation}, they were used only to verify the correctness of algorithms and/or workflows for some form of Bayesian computation, while we proposed a notion of strong calibration that is ambivalent to any particular statistical framework;
\item \Cref{def: strong} is sufficiently precise to allow for logical deduction, such as proving the strong calibration property holds for a non-traditional learning procedure such as that in \Cref{ex: iterative methods}.
\end{enumerate}

\vspace{5pt}

The main drawback with \Cref{def: strong} appears to be practical, since testing for strong calibration in principle requires access to the \ac{cdf} of $f_\#\mu(\mu_0, y)$ for at least one test function $f \in \mathcal{F}_\Theta$.
In some cases the \ac{cdf} will be explicitly available or easily approximated, but in other cases it will not.
Therefore, in the next section we propose a second, strictly weaker notion of calibration which can be tested without access to the \ac{cdf}.

\subsection{Weakly Calibrated Learning Procedures} \label{subsec: weak calib}

Testing whether a learning procedure is strongly calibrated may be challenging in practice.
Furthermore, as discussed in \cref{subsec: strong calib}, the requirement that both $\mu_0$ and the learning procedure are regular in the sense of \Cref{def: reg def,def:reg_learning_proc} will often be too strong, given the diverse algorithms for uncertainty quantification that have been proposed in literature.
We therefore propose a second, weaker definition that requires neither additional structure to define a \ac{cdf} nor regularity of the distributions involved:

\begin{definition}[Weakly Calibrated] \label{def:calibrated}
Let $B \subseteq \mathcal{P}(\Theta)$ denote a set of belief distributions and $P$ a data-generating model.
A learning procedure $\mu$ is said to be \emph{weakly calibrated} to $(B, P)$ if either of the following equivalent properties hold:
\begin{enumerate}[label=(\roman*)]
	\item $\displaystyle \iint \mu(\mu_0,y) \, \wrt P_\theta (y) \, \wrt\mu_0(\theta) = \mu_0$. \label{def:calibrated:distributional}
	\item $\displaystyle \theta \mapsto \int \mu(\mu_0,y) \, \wrt P_{\theta}(y)$ is a $\mu_0$-invariant Markov kernel on $\Theta$.
	\label{def:calibrated:markov}
\end{enumerate}
for all $\mu_0 \in B$.
If the set $B$ contains a single element, $\mu_0$, we say simply that $\mu$ is weakly calibrated to $(\mu_0,P)$.
\end{definition}

\noindent To give some intuition, the definition \ref{def:calibrated:distributional} above states that if one randomises the true parameter according to $\theta \sim \mu_0$, generates synthetic data according to $y \sim P_\theta$, and then samples $\vartheta \sim \mu(\mu_0,y)$ from the distributional output, this should be identical in distribution to sampling $\vartheta$ from $\mu_0$ directly.
Similarly to \Cref{rem: quant strong}, one could consider quantifying departures from weak calibration in terms of a statistical divergence between the two measures appearing in \ref{def:calibrated:distributional}, but here we focus on testing for equality and quantitative descriptions will not be pursued.
Focussing on \ref{def:calibrated:markov}, note that a sufficient condition is provided by the \textit{detailed balance} condition \citep[Eq.~20.5 in][]{meyn2009markov}
\begin{equation}
\mu_0(\rd \theta) \int \mu(\mu_0,y)(\rd \vartheta) \, \rd P_\theta(y) = \mu_0(\rd \vartheta) \int \mu(\mu_0,y)(\rd \theta) \, \rd P_\vartheta(y), \qquad \forall \theta, \vartheta \in \Theta \label{eq: DB} .
\end{equation}
On the other hand, the existence of non-reversible Markov kernels that are $\mu_0$ invariant \citep[e.g.][]{bierkens2016non} demonstrates that \eqref{eq: DB} is not a necessary condition for \ref{def:calibrated:markov} to hold.

The main practical advantage of \Cref{def:calibrated} is that we may test whether a learning procedure is weakly calibrated without access to \ac{cdf}s of any univariate summary $f_\# \mu(\mu_0,y)$, $f \in \mathcal{F}_\Theta$:

\begin{remark}[Testing whether a Learning Procedure is Weakly Calibrated] \label{rem: GOF2}
Let $\mu_0 \in \mathcal{P}(\Theta)$ and let
\begin{align*}
\theta_i & \stackrel{\text{\ac{iid}}}{\sim} \mu_0 \\
y_i \mid \theta_i & \stackrel{\text{\ac{iid}}}{\sim} P_{\theta_i}  \\
\vartheta_i \mid \theta_i, y_i & \stackrel{\text{\ac{iid}}}{\sim} \mu(\mu_0,y_i) .
\end{align*}
Then weak calibration of a learning procedure $\mu$ to $(\mu_0,P)$ can be tested using any goodness-of-fit test for the null hypothesis that $\vartheta_i \stackrel{\text{\ac{iid}}}{\sim} \mu_0$.
Alternatively if $\mu_0$ and $\mu$ are each regular, one could instead test for weak calibration by picking one or more functions $f \in \mathcal{F}_\Theta$ and using any goodness-of-fit test for the null hypothesis 
\begin{align*}
F_{f_\# \mu_0}(f(\vartheta_i)) \stackrel{\text{\ac{iid}}}{\sim} \mathcal{U}(0,1) .
\end{align*}
This is of course equivalent to the procedure described in \Cref{rem: GOF2} provided a sufficiently large set of $f \in \mathcal{F}_r(\Theta)$ are used, but we write it in this way to draw a comparison with \Cref{rem: GOF}.
\end{remark}

\subsubsection{Examples of Weakly Calibrated Learning Procedures}
\label{subsubsec: examples weak}

A natural question is whether a learning procedure that is strongly calibrated to $(B, P)$ is also weakly calibrated to $(B,P)$, as the nomenclature suggests.
This is indeed the case, as stated below and proven in \Cref{sec:proof:S_implies_W}.

\begin{lemma}[Strongly Calibrated $\implies$ Weakly Calibrated] \label{lem:S_implies_W}
Let $\Theta = \reals^d$ for some $d \in \naturals$.
Suppose that $\mu$ is a regular learning procedure that is strongly calibrated to $(B,P)$, where $B \subseteq \mathcal{P}_r(\Theta)$ and $P$ is a data-generating model.
Then the learning procedure $\mu$ is also weakly calibrated to $(B,P)$.
\end{lemma}

\noindent By virtue of \Cref{lem:S_implies_W}, the learning procedures that were shown to be strongly calibrated in \Cref{subsubsec: examples strong} are also weakly calibrated.
However, the converse is not true in general, and the following example provides a cautionary tale:

\begin{figure}[t!]
\includegraphics[width = \textwidth]{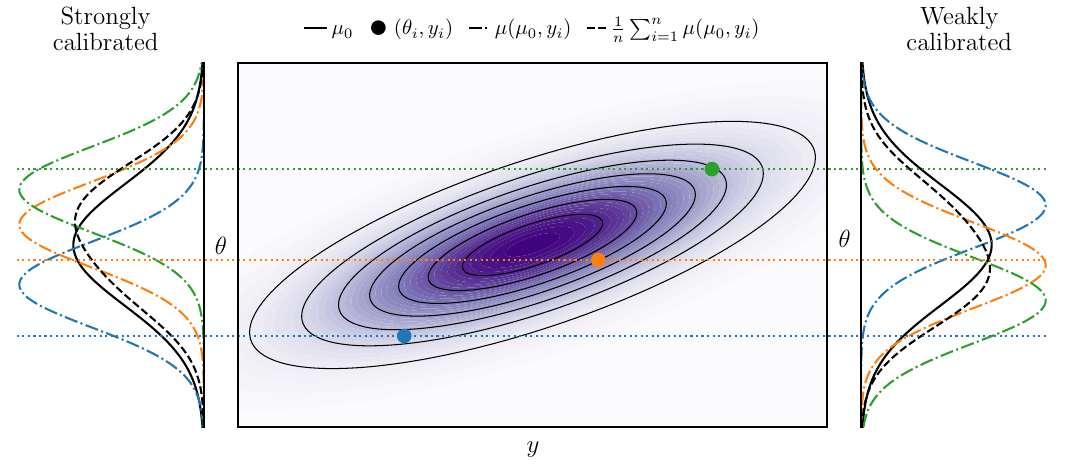}
\caption{\protect\input{figures/caption}}
\label{fig: illustration}
\end{figure}

\begin{example}[Weakly Calibrated $\centernot\implies$ Strongly Calibrated]
\label{ex: weak is bad}
A learning procedure $\mu$ may produce quite unreasonable distributional output $\mu(\mu_0,y)$ and yet be weakly calibrated.
As a concrete example, consider $\Theta = \reals$, an initial belief distribution $\mu_0 = \mathcal{N}(0,1)$, and a data-generating model $P_y(\theta)$ distributed according to $y = \theta + \epsilon$, with independent noise $\epsilon \sim \mathcal{N}(0,1)$.
The Bayesian learning procedure produces $\mu(\mu_0,y) = \mathcal{N}(y/2,1/2)$ and is both weakly and strongly calibrated to $(\mu_0, P_y)$ (see the left hand panel in \Cref{fig: illustration}).
The ``mirror Bayes'' learning procedure, which flips the sign of the datum $y$ before the Bayesian learning procedure is applied, produces $\mu(\mu_0,y) = \mathcal{N}(-y/2,1/2)$, which is not strongly calibrated to $(\mu_0, P_y)$ but is nevertheless weakly calibrated to $(\mu_0, P_y)$ (see the right hand panel in \Cref{fig: illustration}).
\end{example}

\noindent Thus there is a trade-off between strong and weak calibration, where the more straight-forward approach to testing afforded by weak calibration occurs at the expense of failing to rule out pathologically bad learning procedures, such as \Cref{ex: weak is bad}.

An important class of learning procedures that are widely used and yet are \emph{not} weakly calibrated are the generalised Bayesian learning procedures \citep{Bissiri2016}.
These are typically not weakly calibrated to the data-generating model and the prior, since these learning procedures are motivated by the \emph{M-open} setting \citep[\S6.1.2]{Bernardo2000} where the data-generating model may be misspecified.
A canonical example of a generalised Bayesian procedure is presented  next:

\begin{example}[Fractional Posteriors are not Weakly Calibrated] \label{eg:power}

To avoid technical obfuscation, in this example we abuse notation and assume that $\mu_0$ and $\mu(\mu_0,y)$ can be identified with densities with respect to the reference measure $\lambda$ on $\Theta$, i.e.\ $\mu_0(A) = \int_A \mu_0(\theta) \, \wrt \lambda(\theta)$ for each $\mu_0$-measurable set $A$ (and analogously for $\mu(\mu_0,y)$).
Similarly, we assume that $P_\theta$ admits a density $p(\cdot \mid \theta)$ with respect to a suitable reference measure $\rd y$ on $Y$.\footnote{Note that this is not the same as assuming $\mu$ and $\mu_0$ are regular, since their \ac{pdf}s are not required to be positive on $\Theta$.}

Here we consider \textit{fractional posteriors} \citep{Bhattacharya2019}, a prototypical instance of a generalised Bayesian learning procedure.
As with partial posteriors in \cref{ex: partial strong}, fractional posteriors have been proposed as a remedy for model misspecification (e.g.\ in SafeBayes, \cite{Grnwald2017}).
The distributional output of a fractional posterior is defined as
$\mu(\mu_0,y)(\theta) \defeq p(y|\theta)^t \mu_0(\theta) / p_t(y)$, $\theta \in \Theta$,
where $t \in [0,1]$ and we have defined $p_t(y) \defeq \int p(y \mid \vartheta)^t \mu_0(\vartheta) \, \wrt\vartheta$, assuming that $p_t(y) > 0$.
As an example, consider $\mu_0 = \mathcal{N}(0, 1)$, $p(y \mid \theta) = \mathcal{N}(y; \theta, \sigma^2)$, $\sigma > 0$.
Our aim is to verify condition (i) in \Cref{def:calibrated}, which requires the distribution
\begin{align*}
	\iint \mu(\mu_0,y) \, \wrt P_\theta (y) \, \wrt\mu_0(\theta)
	= \mathcal{N}\left(
			\vartheta;
			0, \frac{t^2 (\sigma^2 + 1) + \sigma^2(t + \sigma^2)}{(t + \sigma^2)^2} 	\right)
\end{align*}
to be equal to $\mathcal{N}(\vartheta; 0, 1)$, i.e.
\begin{align*}
	\frac{t^2(\sigma^2 + 1) + \sigma^2(t + \sigma^2)}{(t + \sigma^2)^2}= 1
	\implies t^2 (\sigma^2 + 1) + \sigma^2 (t + \sigma^2) &= (t + \sigma^2)^2
	\implies \sigma^2 t (t - 1) = 0.
\end{align*}
Thus, fractional posteriors are weakly calibrated if and only if either $t=1$, which reduces to standard Bayesian inference (\Cref{ex: Bayes strong}), or $t = 0$, which is data-agnostic (\Cref{eg:trivial strong}).
\end{example}

Finally we present two examples of learning procedures that are neither strongly nor weakly calibrated, to demonstrate the potential consequences of methods not being calibrated.

\begin{figure}[t!]
\includegraphics[width = \textwidth]{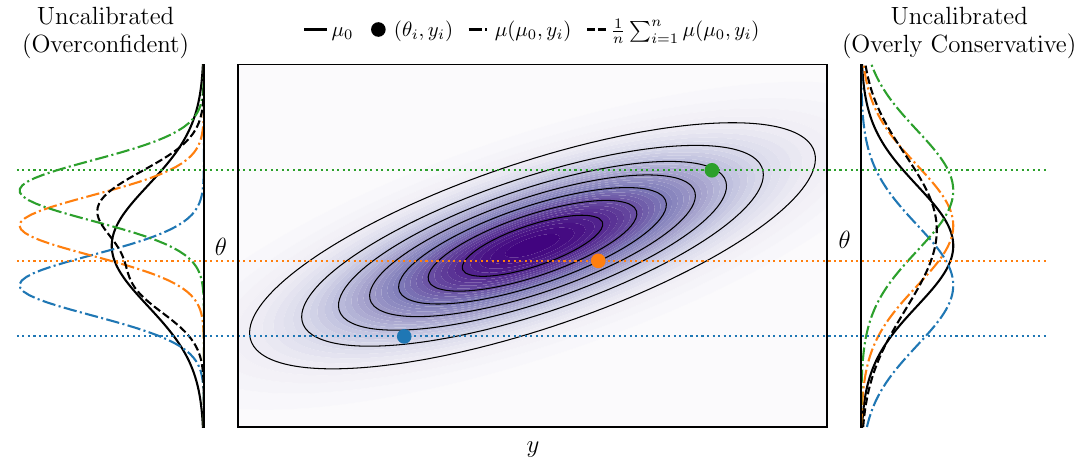}
\caption{\protect\input{figures/caption_uncalibrated}}
\label{fig: illustration_uncalibrated}
\end{figure}

\begin{example}[Consequences of Uncalibrated Methods] \label{eg: uncalibrated}

	We return to the setting of \cref{ex: weak is bad}.
	Recall that we have an initial belief distribution $\mu_0 = \mathcal{N}(0, 1)$ and a data-generating model $P_y(\theta)$ such that $y = \theta + \epsilon$ with independent noise $\epsilon \sim \mathcal{N}(0,1)$.
	The Bayesian learning procedure $\mu(\mu_0, y) = \mathcal{N}(y / 2, 1/2)$ is both weakly and strongly calibrated to $(\mu_0, P_y)$.

	Consider the setting of learning procedures that return distributional output $\tilde{\mu}(\mu_0, y) = \mathcal{N}(y / 2, \eta / 2)$ for some $\eta > 0$, that is, the procedures have the same mean as the Bayesian learning procedure but a different variance for $\eta \neq 1$.
	We illustrate the output in \cref{fig: illustration_uncalibrated}.
	When $\eta = 0.5 < 1$ (left panel), the learning procedures are \emph{overconfident}.
	The output $\tilde{\mu}(\mu_0, y)$ is narrower and more peaked than the correctly specified Bayesian procedure $\mu(\mu_0, y)$, with the consequence that the true parameter $\theta_i$ typically lies further in the tails of the distribution than for the correctly specified procedure. 
	Thus, the misspecified procedure will often suggest a high degree of confidence in the wrong value of the parameter $\theta_i$. 

	Conversely, when $\sigma = 2 > 1$ (right panel), the learning procedures are \emph{overly conservative}.
	The procedure $\tilde{\mu}(\mu_0, y)$ produces a distributional output that is wider and flatter than the correctly specified Bayesian procedure $\mu(\mu_0, y)$.
	Thus the true value of the parameter $\theta_i$ will typically be closer to the mean than the posterior variance would suggest, with the consequence that a user will often associate an accurate estimator of $\theta$ with a high degree of uncertainty.
	In both cases $\sigma = 0.5$ and $\sigma = 2$ note that the average of $\mu(\mu_0, y_i)$ differs from $\mu_0$.

\end{example}

\subsubsection{Relation to Existing Concepts} \label{subsubsec: related weak}

Here we compare and contrast our notion of weak calibration with concepts appearing in earlier work and in related fields.

\paragraph{Forecast Assessment:}

Our notion of weak calibration is closely related to a concept developed in the literature on forecast assessment, for which the term \emph{marginal calibration} is used  \citep{gneiting2007probabilistic}.
As previously mentioned in \Cref{subsubsec: strong existing}, an important distinction between forecast assessment and the present paper is the sense in which notions such as probabilistic calibration and marginal calibration are applied.
This leads to major differences between forecast assessment and the present work.
For example, probabilistic calibration does not imply marginal calibration in the context of forecast assessment,\footnote{A simple example of a forecaster who is marginally calibrated but not probabilistically calibrated is provided by the \textit{unfocussed forecaster} of \cite{gneiting2007probabilistic}; see also \cite{hamill2001interpretation}. These examples have no analogue in our context, due to the fact that there is no inherent randomness in the ``true'' parameter $\theta$, while the quantity being predicted is inherently random in the setting of forecast assessment. } while our notion of strong calibration \textit{does} imply weak calibration in the context of testing whether learning procedures are calibrated, as established in \Cref{lem:S_implies_W}.

\paragraph{Validation of Algorithms for Bayesian Computation:}

The invariance property that underpins our notion of weak calibration has previously been noted in the Bayesian context.
\cite{Talts2018} call this ``\textit{self-consistency of the data-averaged posterior}''.
It appears to have been first used in \cite{geweke2004getting}, who proposed to use it to check the correctness of \ac{mcmc} algorithms and their code.
Therein, the author proposed to alternatively sample from $y \mid \theta$ and $\theta \mid y$, the latter using \ac{mcmc}.
For a correctly implemented \ac{mcmc} method, $\theta$ will be marginally distributed according to the prior $\mu_0$ after an initial burn-in period has passed.
\citet{geweke2004getting} performed a collection of univariate hypothesis tests for this weak calibration null hypothesis, followed by a Bonferroni correction to adjust for multiple testing.
Our \Cref{def:calibrated} is similar in spirit, but is precise enough to permit logical deduction, such as \Cref{lem:S_implies_W}, and yet general enough to cover learning procedures which need not exist within a Bayesian context.
Additionally, we do not assume the structure of \ac{mcmc} that is required to render this Gibbs-like approach practical.

\vspace{5pt}

This completes our formal discussion of what it means for a learning procedure to be called ``calibrated''.
The next section presents several vignettes designed to illustrate our the general framework.

\section{Vignettes} \label{sec: simulations}

In this section we exploit our framework to test whether or not several popular learning procedures are calibrated, with five separate vignettes presented. 
The first two vignettes, \Cref{subsec: Gauss,subsec: ABC}, consider learning procedures that are motivated as being approximations to Bayesian inference and are widely used: Gaussian approximations to non-Gaussian posteriors and \textit{approximate Bayesian computation}, respectively.
In challenging applications, the output produced using these approximations can fail to resemble the usual Bayesian posterior; we therefore view these approximations as learning procedures in their own right and we ask whether their distributional output is calibrated.
\Cref{subsec: ODE calib} presents a topical application to recently developed probabilistic \ac{ode} solvers.
\Cref{subsec: data-driven} concerns the challenge of performing a goodness-of-fit test for strong calibration in multiple dimensions, where a suitable test function $f$ must first be identified.
The final vignette, \Cref{sec: robust} examines how our notions of calibration can be extended to the setting where the data-generating model is misspecified.

\subsection{Gaussian Approximations} \label{subsec: Gauss}


A common approach in statistics is to output a Gaussian distribution which approximates, in some sense, the distributional output of an idealised learning procedure.
The targeted learning procedure will often be Bayesian inference, however Gaussian approximations can also be used within different inferential paradigms.
As an example of such an approach, Gaussian approximations are often the output of variational inference methods, wherein the learning procedure outputs the member of a family of distributions (in this case Gaussian) which minimises a divergence from the target distribution \citep{Blei2017}.
A distinct but related approach is that of fitting a Gaussian approximation based on only local information.
The Laplace approximation, which outputs a Gaussian distribution centred at a maximum of the log density of the target distribution and with covariance equal to the inverse of the Hessian of the log density at this point, is a canonical example of such a method.

As a first simulation study we test the calibration of Laplace approximations to the Bayesian posterior in a model with a location parameter $\theta$.
We assign a prior $\mu_0 = \mathcal{N}(0, 1)$, and a Student's $t$ data-generating model $P_\theta$ such that $y$ consists of $n$ independent draws from a $\mathcal{T}(\theta, 1, \nu)$ distribution.
To be specific, $y = (y^{(n)})_{n=1}^N$, with $y^{(n)} \stackrel{\text{\ac{iid}}}{\sim}\mathcal{T}(\theta, 1, \nu)$ for $n \in \lbrace 1, \dots, N\rbrace$.

The true posterior in this case is non-Gaussian and so our expectation is that a Laplace approximation will be neither strongly nor weakly calibrated.
However, for $\nu \to \infty$ or $N \to \infty$ (and $\nu > 2$) the posterior will become increasingly close to Gaussian, in the former case due to the Student's $t$ distribution becoming increasingly close to Gaussian as $\nu \to \infty$, and in the latter due to the asymptotic normality of the posterior as $N \to \infty$ by the Bernstein--von Mises theorem for $\nu > 2$.
We therefore would expect it to be increasingly challenging for the tests in \Cref{rem: GOF} and \Cref{rem: GOF2} to reject respectively strong and weak calibration as $\nu \to \infty$ or $N \to \infty$.

\begin{figure}[t]
	\begin{subfigure}{0.49\textwidth}
	\includegraphics[width=\textwidth]{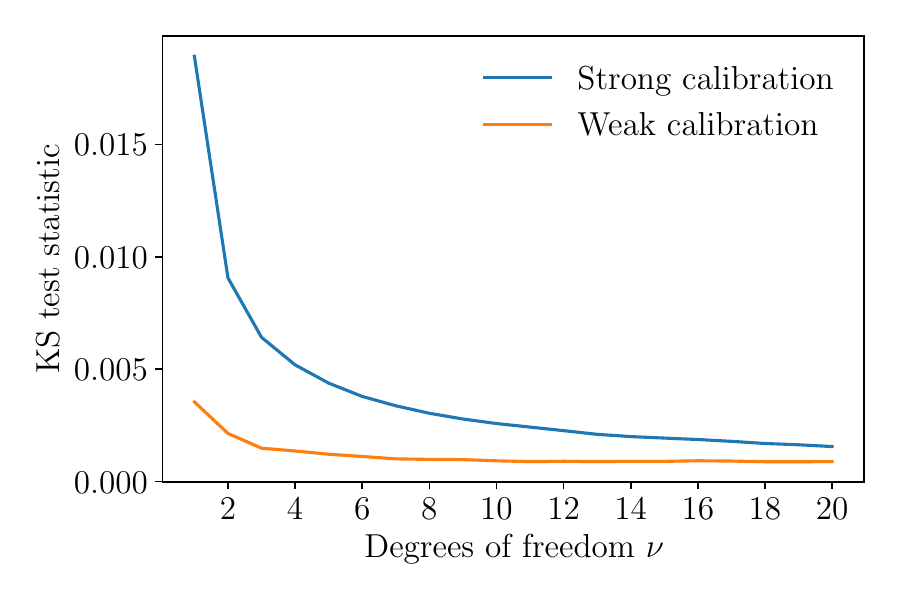}
	\end{subfigure}
	\begin{subfigure}{0.49\textwidth}
	\includegraphics[width=\textwidth]{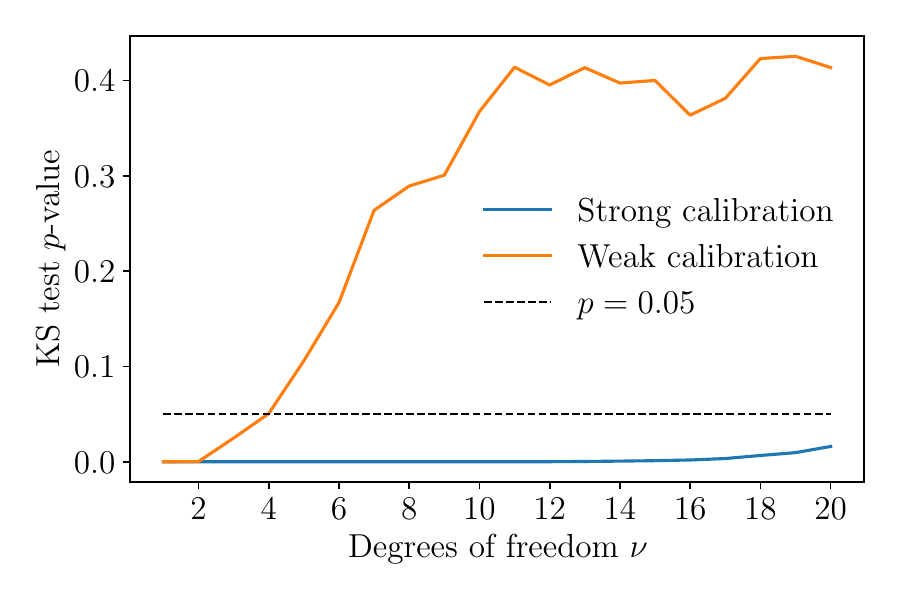}
	\end{subfigure}
	\\
	\begin{subfigure}{0.49\textwidth}
	\includegraphics[width=\textwidth]{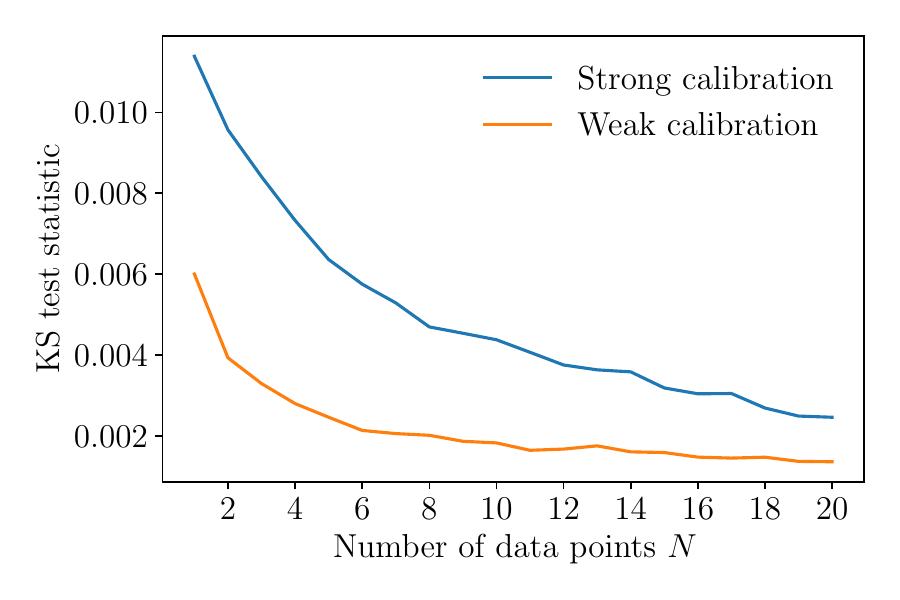}
    \end{subfigure}
	\begin{subfigure}{0.49\textwidth}
	\includegraphics[width=\textwidth]{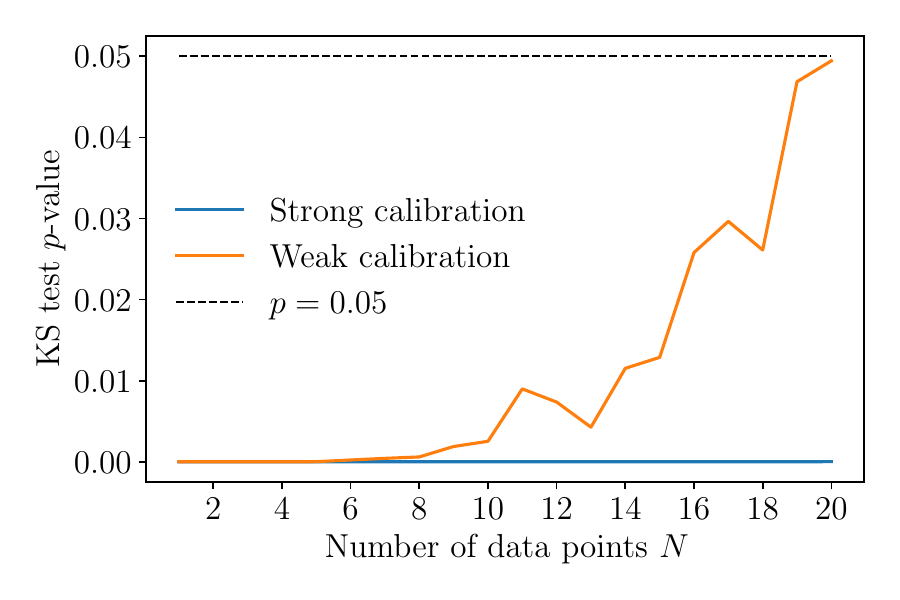}
	\end{subfigure}
	\caption{Gaussian approximations: \ac{ks} test statistics (left) and $p$-values (right) for strong and weak calibration of Laplace approximations in the $t$-distribution example for varying degrees of freedom $\nu$ and $N = 5$ (top) and varying number of data $N$ and $\nu = 3$ (bottom).} \label{fig:laplace_calibration}
\end{figure}

In univariate cases such as this, we may employ the identity test function $f(\theta) = \theta$ and a one-sample \ac{ks} test to check for uniformity in the tests in \Cref{rem: GOF,rem: GOF2}.
Laplace approximations were computed for $10^6$ realisations from the hierarchical model $\theta_i \stackrel{\text{\ac{iid}}}{\sim} \mu_0$, $y_i \mid \theta_i \stackrel{\text{\ac{iid}}}{\sim} P_{\theta_i}$, for each of $\nu \in \lbrace 1, 2, \dots 20\rbrace$ with $N=5$ and for each of $N \in \lbrace 1, 2, \dots 20\rbrace$ with $\nu = 3$.
The strong and weak calibration test results are summarised in \Cref{fig:laplace_calibration}.
As expected, we see that the power of both the strong and weak calibration tests decrease as $\nu$ and $N$ increase, with the \ac{ks} test statistics (defined in \eqref{eq:Kolmogorov_distance}) showing decreasing departures from uniformity.
While the strong calibration test rejects the null hypothesis at a 0.05 significance criterion for all values of $\nu$ and $N$ tested, the weak calibration test fails to reject at a 0.05 level for most of the $\nu$ range.
However, for the results with varying $N$, we see that weak calibration test correctly rejects the null hypothesis at a 0.05 significance level up to $N = 20$.

A test of strong calibration is clearly preferable to a test of weak calibration in situations where it is possible to be performed.
However, these results indicate that the weaker test in \Cref{rem: GOF2} is still able to provide a useful check of calibration in some situations, with the benefit of being simpler to compute and more widely applicable than the test in \Cref{rem: GOF}.

\subsection{Approximate Bayesian Computation} \label{subsec: ABC}

Performing Bayesian inference in settings for which the data-generating model $P_\theta$ does not have a tractable \ac{pdf} is challenging, with \ac{abc} methods \citep{beaumont2002approximate} often used as an alternative in such situations.
The key idea in \ac{abc} is that, in contrast to the standard Bayesian procedure of conditioning on the observed dataset $y = y_{\textrm{obs}}$, one instead conditions on the event that $d(y, y_{\textrm{obs}}) < \epsilon$, for some distance $d \colon Y \times Y \to [0,\infty)$ and some \emph{tolerance} $\epsilon > 0$.
Typically the distance is specified by embedding the data into a finite-dimensional normed vector space $S$ via a \emph{summary statistic} function $s \colon Y \to S$ and specifying the distance as $d(y, y_{\textrm{obs}}) = \|s(y) - s(y_{\textrm{obs}})\|$.

As a consequence of \Cref{ex: partial strong}, the learning procedure that exactly conditions on $s(y) = s(y_{\textrm{obs}})$, i.e.\ \ac{abc} with tolerance $\epsilon = 0$, is guaranteed to be strongly calibrated.
Likewise in the limit of $\epsilon \to \infty$ the \ac{abc} posterior will be strongly calibrated, as the posterior will collapse to the prior (see \Cref{eg:trivial strong}).
For $\epsilon \in (0, \infty)$ the \ac{abc} posterior will in general however be neither strongly nor weakly calibrated.
To resolve this lack of calibration of \ac{abc} methods, \citet{fearnhead2012constructing} proposed the \emph{noisy \ac{abc}} algorithm, which is calibrated for any tolerance $\epsilon \geq 0$.
Rather than conditioning on the event $\|s(y) - s(y_{\textrm{obs}})\| < \epsilon$, noisy \ac{abc} replaces $s(y_{\textrm{obs}})$ with noisy summary statistics $\tilde{s}_{\textrm{obs}}$ generated according to
$\tilde{s}_{\textrm{obs}} = s(y_{\textrm{obs}}) + \epsilon x$,
with $x$ uniformly distributed on the unit ball in $S$.
The distributional output of noisy \ac{abc} is the partial posterior based on $\tilde{s}_{\textrm{obs}}$, which takes into account the additional noise in the data-generating model, and is therefore strongly calibrated by an extension of the argument in \Cref{ex: partial strong}.

\begin{figure}[t!]
	\includegraphics[width=\textwidth]{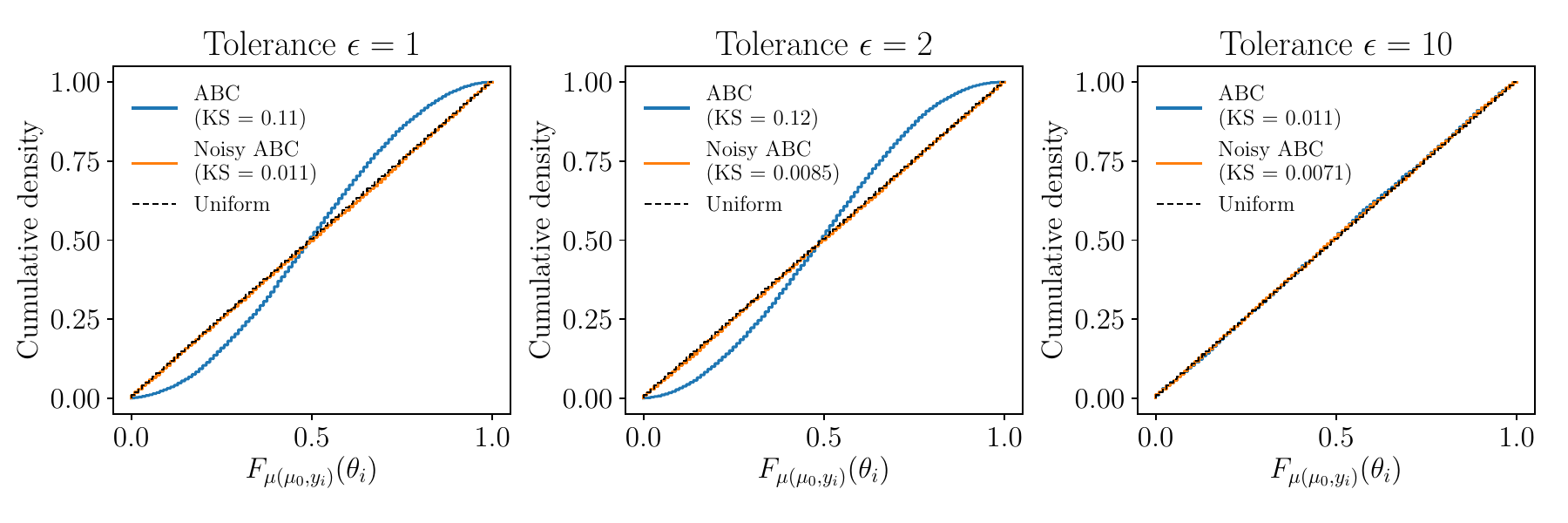}
		\caption{Approximate Bayesian computation: Empirical \ac{cdf}s for both \ac{abc} and noisy \ac{abc} in the $g$-and-$k$ quantile distribution example, for tolerances $\epsilon = 1$ (left), $\epsilon = 2$ (middle) and $\epsilon = 10$ (right).
		The values of the \ac{ks} test statistics are shown in the legend.
	} \label{fig:abc_calibration_ecdfs}
\end{figure}

Here we consider the parameter inference task for a $g$-and-$k$ distribution.
The $g$-and-$k$ distribution is defined through the inverse of its \ac{cdf} (quantile function) and it does not have a closed-form \ac{pdf}  \citep[though the \ac{pdf} can be evaluated numerically;][]{rayner2002numerical}. 
Here we aim to infer the location parameter $\theta$, which is assigned a prior $\mu_0 = \mathcal{N}(0, 1)$, given a dataset $y \in \reals^N$, $N = 20$, generated according to the data-generating model
\begin{align*}
	P_\theta \; : \; y^{(n)} &= \theta + b \left(1 + 0.8 \; \frac{1 - \exp(-g u_n)}{1 + \exp(-g u_n)}\right) u_n (1 + u_n^2)^k , 
\end{align*}
with $u_n \stackrel{\text{\ac{iid}}}{\sim} \mathcal{N}(0, 1)$, $n \in \lbrace 1, 2, \dots N \rbrace$, $b = 1$, $g=2$ and $k = 0.5$.
For the tests that follow we computed independent realisations from the hierarchical model $\theta_i \stackrel{\text{\ac{iid}}}{\sim} \mu_0$, $y_i \mid \theta_i \stackrel{\text{\ac{iid}}}{\sim} P_{\theta_i}$.
In each case, data were summarised as a vector $s \colon \reals^N \to \reals^5$ consisting of the five quartiles of the dataset, and rejection sampling was used to generate $M$ samples $\{\theta_i^m\}_{m=1}^M$ from the distributional output of both \ac{abc} and noisy \ac{abc}, for tolerances $\epsilon \in \lbrace 1, 2, \dots 10\rbrace$.
Single samples ($M = 1$) can be directly used to test for weak calibration, as per \Cref{rem: GOF2}.
However, the intractability of the distributional output for \ac{abc} and noisy \ac{abc} precludes a straightforward test for strong calibration.
Instead, we consider a variant of the test for strong calibration in \Cref{rem: GOF}, which in a similar spirit to \citet{Talts2018}, wherein we test whether the rank statistics $r(\lbrace \theta^m_i \rbrace_{m=1}^M,\theta_i)$ are \ac{iid} samples from the discrete uniform distribution on $\lbrace 0, 1, \dots M \rbrace$.
For testing strong calibration, a total of $10^4$ realisations of the hierarchical model were considered with $M = 100$, while for the less computationally demanding test for weak calibration a total of $10^6$ realisations were considered with $M = 1$.

\begin{figure}[t!]
	\begin{subfigure}{0.5\linewidth}
	\includegraphics[width=\textwidth]{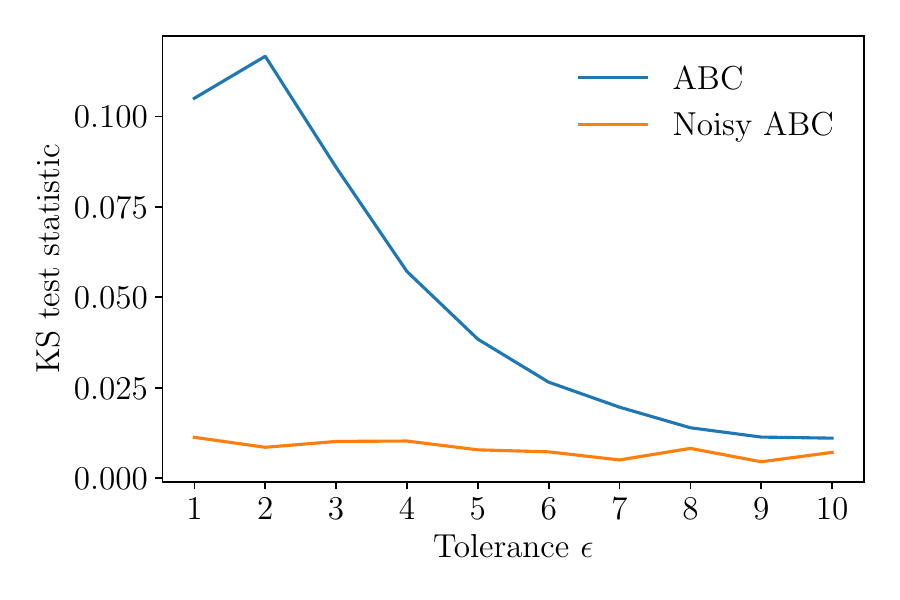}
	\end{subfigure}%
	\begin{subfigure}{0.5\linewidth}
	\includegraphics[width=\textwidth]{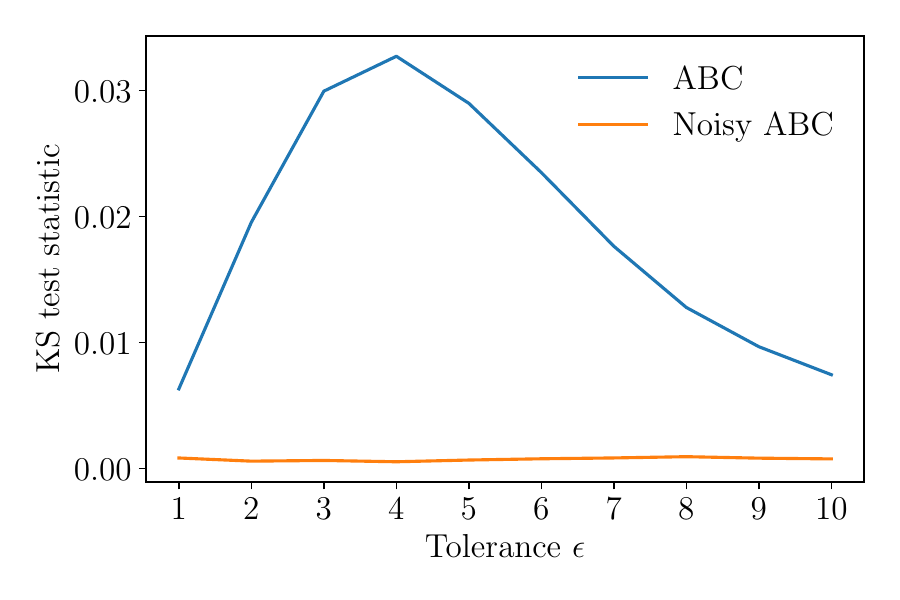}
	\end{subfigure}%
	\\
	\begin{subfigure}{0.5\linewidth}
	\includegraphics[width=\textwidth]{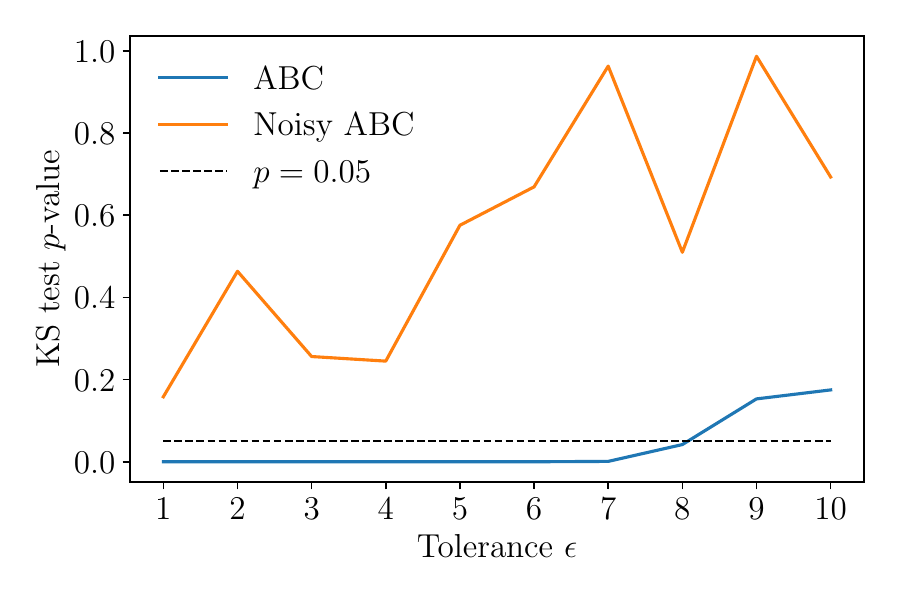}
	\end{subfigure}%
	\begin{subfigure}{0.5\linewidth}
	\includegraphics[width=\textwidth]{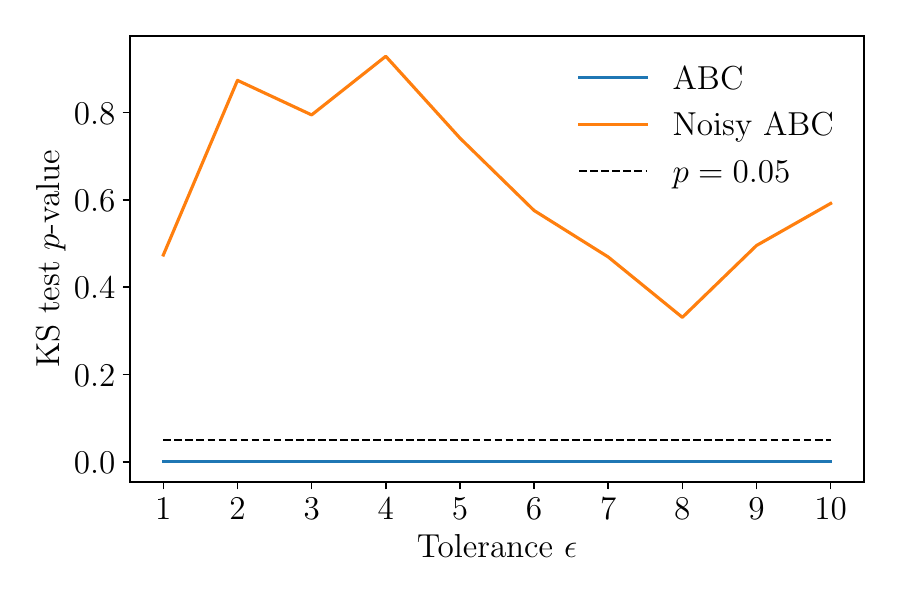}
	\end{subfigure}
	\caption{Approximate Bayesian computation: \ac{ks} test statistics (top) and $p$-values (bottom) for strong (left) and weak (right) calibration of \ac{abc} and noisy \ac{abc} in the $g$-and-$k$ quantile distribution example, for various tolerances $\epsilon > 0$.
	} \label{fig:abc_calibration}
\end{figure}

\Cref{fig:abc_calibration_ecdfs} presents empirical \ac{cdf}s for both \ac{abc} and noisy \ac{abc}, on which our test for strong calibration is based.
\Cref{fig:abc_calibration} presents the \ac{ks} test statistics and corresponding $p$-values for both strong and weak calibration, for different values of the tolerance $\epsilon > 0$.
In each case noisy \ac{abc} is, as expected, seen to be better calibrated than \ac{abc}. Both the strong and weak calibration tests correctly fail to reject the null hypothesis at a 0.05 significance level for noisy \ac{abc}, which is strongly (and weakly) calibrated, for all values of the tolerance $\epsilon$.
The strong calibration test fails to reject the null hypothesis that \ac{abc} is strongly calibrated for the highest two tolerances $\epsilon \geq 9$. The weak calibration test on the other hand correctly rejects at a 0.05 level the null hypothesis that \ac{abc} is weakly calibrated for all $\epsilon$. The apparent greater power of the weak calibration test here likely arises from the much larger number of model realisations used --- $10^6$ compared to $10^4$ for the strong test --- for a given computational expenditure due to the need to generate only $M=1$ \ac{abc} sample per realisation rather than $M = 100$. A final interesting point of note is that both weak and strong calibration show a ``dip'' in the \ac{ks} test statistic at $\epsilon=1$, reflecting that as $\epsilon\to 0$ classical \ac{abc} tends towards a Bayesian procedure, which is guaranteed to be calibrated.

\subsection{Calibration of Probabilistic ODE Solvers} \label{subsec: ODE calib}

A traditional (adaptive) numerical method for the approximate solution of an \ac{ode} accepts, as its input, an error tolerance $\tau > 0$ and returns, as its output, an approximation to the solution of the \ac{ode}.
In general it is not guaranteed 
that the resulting approximation has error less than $\tau$, but empirical analysis over a range of typical \acp{ode} can provide reassurance that the error will be below $\tau$ for many problems practically encountered.
In contrast to the traditional approach, there has been a concerted research effort in recent years to develop \acp{pnm} for \acp{ode}.
A \ac{pnm} returns a probability distribution over the solution space of the \ac{ode}, representing epistemic uncertainty associated with the unknown true solution of the \ac{ode}.
The scale of this distributional output can be used as the basis for selecting a suitable time step size in order to drive the uncertainty below a user-specified tolerance $\tau$, if desired.
Compared to traditional numerical methods, which have benefited from over a century of development, important questions regarding their behaviour of \acp{pnm} remain unanswered, including whether such methods are calibrated.
Most \acp{pnm} exploit \ac{gp} models for the solution of the \ac{ode}, 
motivated by mathematical convenience rather than detailed knowledge of the \ac{ode} to be solved.
These models typically include hyperparameters for the \ac{gp},
which are jointly estimated along with the solution of the \ac{ode}.
Given that \ac{pnm} act on the basis of a default \ac{gp} model, essentially independent of initial belief $\mu_0$ regarding the \ac{ode} at hand, it is unclear whether hyperparameter estimation is sufficient to ensure \ac{pnm} are calibrated.

\newcommand\figwidth{0.191}

\begin{figure}[t!]
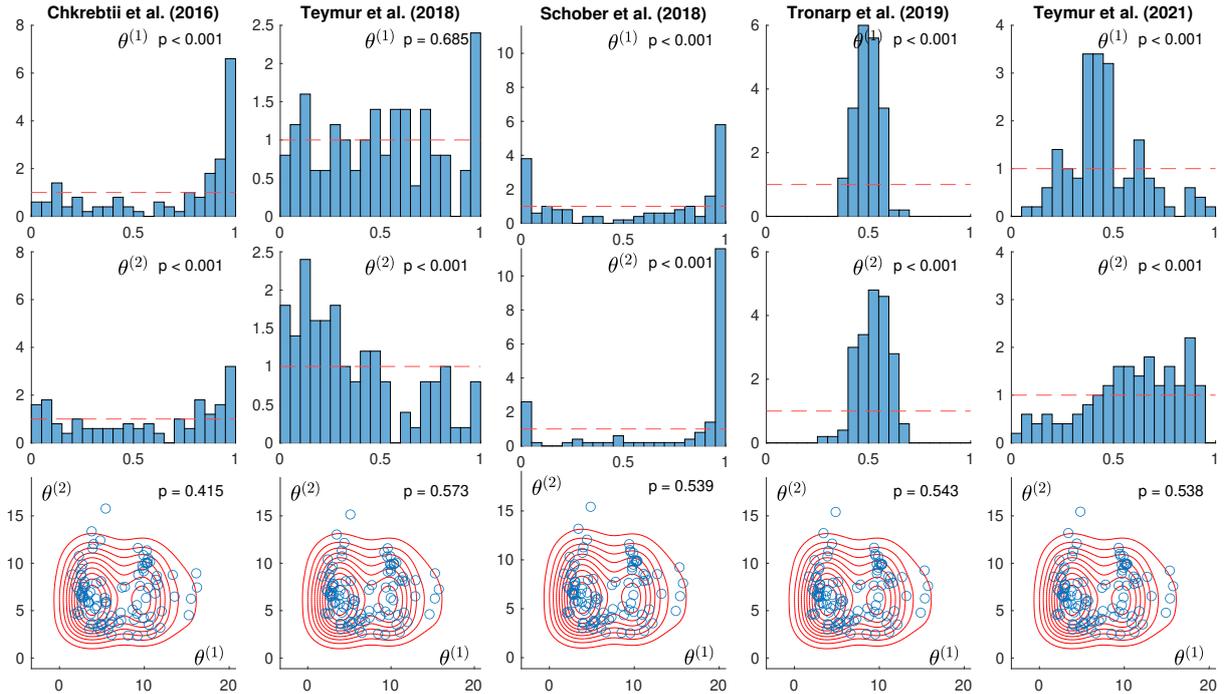

\includegraphics[width = \figwidth\textwidth]{figures/pnodes/\converteps{chk}}\ \ 
\includegraphics[width = \figwidth\textwidth]{figures/pnodes/\converteps{teym}}\ \ 
\includegraphics[width = \figwidth\textwidth]{figures/pnodes/\converteps{scho}}\ \ 
\includegraphics[width = \figwidth\textwidth]{figures/pnodes/\converteps{tron}}\ \ 
\includegraphics[width = \figwidth\textwidth]{figures/pnodes/\converteps{bbpn}}
\caption{ 
Calibration of probabilistic ODE solvers: 
Samples (blue) from the strong calibration test statistic $F_{f_\# \mu(\mu_0,y_i)}(f(\theta_i))$ (c.f. \Cref{rem: GOF}), with $f(\theta) = \theta^{(1)}(10)$ (top row), $\theta^{(2)}(10)$ (middle row), and from the weak calibration test statistic $\vartheta_i \sim \mu(\mu_0,y_i)$ (bottom row; c.f. \Cref{rem: GOF2}).
The reference distribution, corresponding to the null hypothesis that the learning procedures are calibrated, is in each case shown in red. 
The $p$-values for the associated hypothesis tests (see main text for details) are shown in the top right-hand corner of each panel.
}
\label{fig: ODEs}
\end{figure}

The principal application of \acp{pnm} for \acp{ode} is to \textit{inverse problems}, where an \ac{ode}'s parameters are to be estimated based on a dataset.
This usually requires the numerical solution of many \acp{ode}, each corresponding to different values of the parameters, to see which parameter values are compatible with the dataset.
The motivation for \acp{pnm} in this setting is that the solution of the \acp{ode} can be viewed as an unknown latent quantity and integrated out, potentially using a fast-but-crude \ac{pnm} in place of an adaptive \ac{ode} solver and adjusting credible sets for \ac{ode} parameters in a way commensurate with the accuracy of the \ac{pnm} used.
However, the success of this approach hinges on whether the underlying \ac{pnm} is calibrated, as otherwise under- or over-confident parameter inferences could be produced.
To shed light on this question, we considered the probabilistic numerical solution of the following Lotka--Volterra \ac{ode}
\begin{equation}
\frac{\mathrm{d}\theta}{\mathrm{d}t} = \left[  \begin{array}{c} \alpha \theta^{(1)} + \beta \theta^{(1)} \theta^{(2)} \\ \gamma \theta^{(2)} + \delta \theta^{(1)} \theta^{(2)} \end{array} \right],\quad \theta(0) = \left[ \begin{array}{c} 10 \\ 10 \end{array} \right] ,\quad  t \in [0,10] , \label{eq: LV}
\end{equation}
with an initial belief distribution $\mu_0$ induced over the solution space of differentiable functions $\theta(t)$ on $[0,10]$ by sampling parameters $(\alpha,\beta,\gamma,\delta)$ from a probability distribution $\pi$ on $[0,\infty)^4$. 
For this experiment we took the distribution $\pi$ to be
\begin{equation*}
	\alpha,\gamma \overset{\mathrm{iid}}{\sim} \mathrm{logNormal}(0,0.25) , \qquad
	\beta,\delta \overset{\mathrm{iid}}{\sim} \mathrm{logNormal}(-2,0.1) ,
\end{equation*}
which produces a variety of periodic trajectories typically associated with this type of predator-prey model.
The following \acp{pnm} were considered: \citet{chkrebtii16}, which employs a particle-based approach requiring parallel simulations to produce empirical credible sets; \citet{teymur18}, which is based on stochastic perturbation of traditional numerical methods, continuing a line of work that originated in \citet{conrad16}; \citet{schober18} and \citet{Tronarp18}, which are both based on Gaussian filtering but with different approaches to the (local) linearisation of \eqref{eq: LV}; and \citet{teymur21}, which is based on a probabilistic version of Richardson extrapolation.
Each method has user-defined settings that can in principle affect the selection of its hyperparameters, and thus, its calibration in the senses used in this paper; for this experiment we considered one setting only for each \ac{pnm}, with full details contained in \Cref{app: PN}.
In particular, default settings were used for some \acp{pnm}, whilst the settings of other \acp{pnm} were manually selected. 
Thus we do not claim to draw general conclusions about the specific \acp{pnm} involved; our aim is only to show how diverse algorithms can be analysed using the notions of calibration we have introduced.

Tests of strong and weak calibration were performed, in each case using the test functions $f_j(\theta) = \theta^{(j)}(10)$, $j \in \{1,2\}$, i.e. the value of the solution at the final time point.
Results are displayed in \Cref{fig: ODEs}. 
The top two rows show histograms of $F_{f_\# \mu(\mu_0,y_i)}(\theta_i^{(j)}(10))$ for $j \in \{1,2\}$, using 100 samples $\theta_i$ drawn from $\mu_0$. 
A Kolmogorov--Smirnov test of uniformity was then used to test whether the \acp{pnm} are strongly calibrated (c.f. \Cref{rem: GOF}). 
The bottom panels show scatter plots of samples $\vartheta_i(10)$ where $\vartheta_i \sim \mu(\mu_0,y_i)$, overlaid on contours of $\mu_0$ (empirically obtained). 
A kernel two-sample test \citep{gretton12} was performed based on samples from the intractable distribution $\mu_0$ to assess whether the \acp{pnm} are weakly calibrated (c.f. \Cref{rem: GOF2}).
The results of these simulations show that strong calibration is not a property enjoyed by most \ac{pnm} at present. 
The only instance where strong calibration was not emphatically rejected is \citet{teymur18}, for inference of the first component $\theta^{(1)}(10)$. 
It is interesting to note that \citet{teymur18} performs an exhaustive grid search for \ac{gp} hyperparameter estimation, which can require more computation compared to the other \ac{pnm} considered, and this may explain its relative success in this calibration assessment.
The remaining \ac{pnm} perform poorly in different ways, including being over-confident \citep[e.g.][]{schober18} and under-confident \citep[e.g.][]{Tronarp18}.
However, we reiterate that these conclusions will depend on additional user-specified settings, specific to how each \ac{pnm} is implemented.
On the other hand, weak calibration was never rejected, and indeed this was also the case over a much wider variety of algorithm settings (not presented).
This suggests that weak calibration of \acp{pnm}, in as far as this testing framework is concerned, is indeed a weak requirement.


%

\color{black}

\subsection{Data-Driven Goodness-of-Fit Testing for Strong Calibration} \label{subsec: data-driven}

For multivariate parameter inference tasks, where e.g.\ $\Theta = \reals^d$, $d > 1$, it will not be possible in general to identify a single test function $f \in \mathcal{F}_\Theta$ that has power against all alternatives to the strong calibration null.
Indeed, even a simultaneous test using all coordinate functions $f_i(\theta) = \theta^{(i)}$, $i = 1,\dots,d$, does not have power against all alternatives, since a multivariate distribution is not uniquely determined by its univariate marginals.
Nevertheless, the richness of the set $\mathcal{F}_\Theta$ is such that we expect \emph{some} $f \in \mathcal{F}_\Theta$ to yield a test with the power to reject the null hypothesis, due to \Cref{lem:weak_control}.
A strategy to select a suitable test function $f$ is therefore required.

Following a generic approach to goodness-of-fit testing, one way to proceed is to consider splitting the collection of simulated parameter-dataset pairs into two disjoint sets: $\mathcal{S}_1 \defeq \{(\theta_i,y_i)\}_{i=1}^s$, $\mathcal{S}_2 \defeq \{(\theta_i,y_i)\}_{i=s+1}^S$.
The first subset $\mathcal{S}_1$ can be used to identify a suitable test function $f$, after which a goodness-of-fit test can be conducted using $f$ and $\mathcal{S}_2$.
The independence of $\mathcal{S}_1$ and $\mathcal{S}_2$ ensures that a test conducted in this way is valid.
To select a suitable test function, one first identifies a sufficiently small subset $\mathcal{F}_s \subset \mathcal{F}_\Theta$ of test functions and, for each $f \in \mathcal{F}_s$, a univariate goodness-of-fit test is performed using $\mathcal{S}_1$.
The element of $\mathcal{F}_s$ that gives rise to the strongest evidence against the null hypothesis, based on $\mathcal{S}_1$, is selected.
The main advantage of a data-splitting approach is that the selection of $f$ is data-driven, as opposed to $f$ being user-specified.
The role of data to inform the selection of $f$ is anticipated to be increasingly important in higher dimensional settings, $d \gg 1$.
To explore this, we consider now a setting that is, at least notionally, infinite dimensional.

Let $\theta : [0,1] \rightarrow \mathbb{R}$ be a continuous function-valued parameter, so that $\Theta = C(0,1)$ is the set of continuous functions on $[0,1]$.
For $\mu_0$ we consider a hierarchical, non-stationary \ac{gp} of the form $\theta(x) := \sigma(x) g(x)$, $g \sim \mathcal{GP}(0,k)$ with $k(x,x') := \exp(-(x-x')^2/\ell^2)$, $\sigma \sim \nu$ for some distribution $\nu$ to be specified, and for simplicity $\ell = 0.1$ is fixed.
Consider the data-generating model that returns $y = (y^{(1)},\dots,y^{(10)})$, where $y^{(n)} = \theta(x_n)$ and $x_n \sim \mathcal{U}(0,1)$ are independently sampled.
A popular, pragmatic workflow acknowledges the non-stationarity encoded in $\mu_0$ but, for computational convenience, fits instead a stationary, non-hierarchical \ac{gp} of the form $\theta(x) = \sigma_0 g(x)$, where the scalar $\sigma_0$ is estimated using maximum likelihood.
Estimating $\sigma_0$ from data enables the scale of the distributional output to roughly adapt to the scale of the dataset, but this is insufficient to ensure the learning procedure is strongly calibrated \citep{Karvonen2020}.
Our interest here is in whether we can detect failure of strong calibration, and for this purpose we consider a simple form of $\nu$ that sets $\sigma(x) = 1+x$ with probability one.
It can be expected that simplified \ac{gp} regression produces a ``compromise'' value of $\sigma_0$, which leads to under-confident inferences for $\theta(x)$ when $x$ is close to 0 and over-confident inferences when $x$ is close to 1.

For the set of candidate test functions $\mathcal{F}_s$, we consider the evaluation functions $f_x(\theta) := \theta(x)$, indexed by $x \in [0,1]$.
A number, $S$, of parameter-dataset pairs were generated, of which $s = \frac{S}{2}$ were assigned to $\mathcal{S}_1$ and used to identify a promising location $x_* \in [0,1]$ at which to perform a hypothesis test of strong calibration using the held-out $\mathcal{S}_2$.
Since the marginals $(f_x)_{\#} \mu(\mu_0,y)$ are Gaussian, it is natural to use a $\chi_s^2$ test, as per \eqref{eq: cook}.
Thus we select $x_*$ to minimise the $p$-value of a two-sided $\chi_s^2$ test, based on $f_x$ and computed using $\mathcal{S}_1$, over $x \in [0,1]$.
The total number of simulated parameter-dataset pairs $S$ was varied from $10$ to $150$ and, through repeated simulation, the $p$-values of a two-sided $\chi_s^2$ test of strong calibration, based on the estimated $x_*$ and $\mathcal{S}_1$, were computed.
As a baseline, we also computed $p$-values for a user-specified test function centred at $x_b := 0.5$.
In \Cref{fig:gp calibrate} (left) we plot log $p$-values as a function of $x$, for $s = 10$ (top) and $s=150$ (bottom), for one typical realisation of $\mathcal{S}_1$.
These results indicate that values of $x$ close to 0 are likely to provide the most power for our hypothesis test.
Here $x_*$ is indicated as a vertical red line and $x_b$ indicated as a vertical blue line; the identification of a suitable $x_*$ is seen to be easier when the number, $s$,  of simulations available in $\mathcal{S}_1$ is increased.
Finally, in \Cref{fig:gp calibrate} (right) we plot the $p$-values obtained when the $x_*$-based and $x_b$-based  tests are applied to $\mathcal{S}_2$.
To avoid reporting an artefact of the random seed, average log $p$-values are reported, along with standard errors, based on 100 independent realisations of $\mathcal{S}_1$ and $\mathcal{S}_2$.
It is seen that the data-driven goodness-of-fit test (based on $x_*$) is more powerful than the user-specified test (based on $x_b$).

\begin{figure}[t!]
\includegraphics[width = \textwidth]{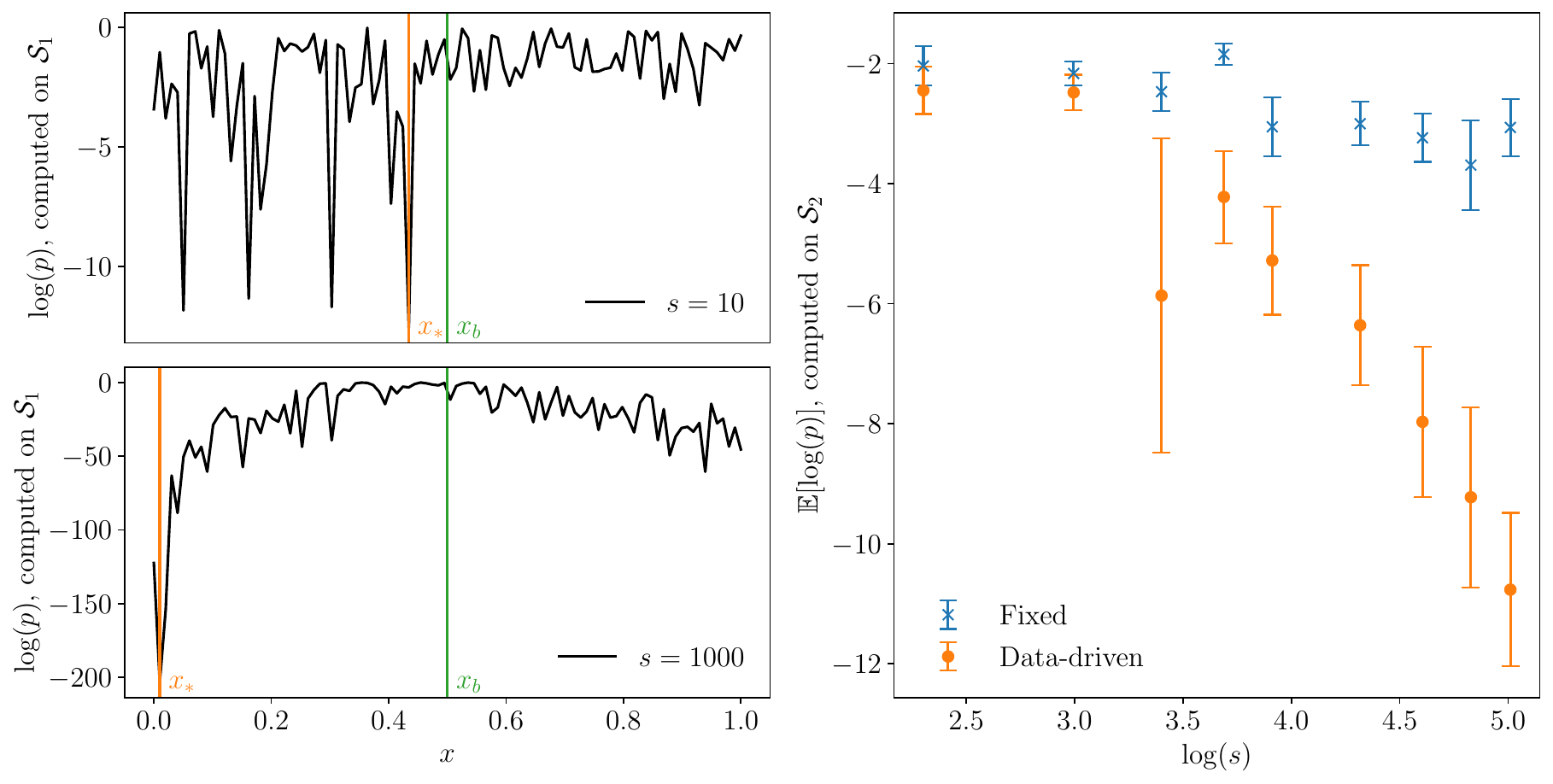}
\caption{Data-driven goodness-of-fit testing for strong calibration:
Here we consider an infinite-dimensional parameter $\theta \in C(0,1)$. On the left, we display typical $p$-values obtained using a test based on the evaluation functional $f_x(\theta) = \theta(x)$ and a number $s$ of independent simulations of the parameter and dataset (top: $s=10$, bottom: $s=150$).
The right hand panel displays average $p$-values obtained when the procedure is applied to $100$ independent realisations of $\mathcal{S}_1$ and $\mathcal{S}_2$, each of size $s$, using either the data-driven choice $x = x_*$ or the fixed choice $x = x_b$ of test.
 }
\label{fig:gp calibrate}
\end{figure}

This illustration makes clear that, for a data-splitting approach to work well, the size of the set $\mathcal{F}_s$ of candidate test functions should be carefully controlled, relative to the number $s$ of samples in $\mathcal{S}_1$.
For example, if we simply took $\mathcal{F}_s = \mathcal{F}_\Theta$, then for any $\alpha \in (0,1)$ there would be infinitely many elements of $\mathcal{F}_s$ for which the null hypothesis is rejected at level $\alpha$ by virtue only of the fact that $\mathcal{S}_1$ is a finite set.
Consideration of multiple data splits can also be exploited to increase the power of such a test \citep{Romano2019}.

\subsection{Robust Calibration} \label{sec: robust}

Our proposed notions of strong and weak calibration can be extended to the \emph{M-open} setting \citep[\S6.1.2]{Bernardo2000} where the data-generating model may be misspecified.
This permits us to define notions of ``robust calibration'', which are analogous (and orthogonal) to the notions of ``robust estimation'' that are already widely studied \citep{Berger1994,huber2009robust}.
For example, suppose that a learning procedure $\mu$ is strongly calibrated to $(\mu_0,P)$.
Then, for any $f \in \mathcal{F}_\Theta$, the distribution $\mathcal{U}_{f,P}$ of the random variable $ F_{f_\# \mu(\mu_0,y)}(f(\theta))$, where $\theta \sim \mu_0$, $y \mid \theta \sim P_\theta$, is by definition $\mathcal{U}(0,1)$.
Thus, when the data-generating model $P$ is misspecified, we may quantify the loss of strong calibration in terms of a statistical divergence between $\mathcal{U}_{f,Q}$ and $\mathcal{U}(0,1)$.

Here we adopt a more practical perspective, using the framework of \Cref{subsec: strong calib} to test the strong calibration null hypothesis in settings where the data-generating model is misspecified.
For example, consider a Bayesian learning procedure $\mu$ for a location parameter $\theta$, which is assigned a prior $\mu_0 = \mathcal{N}(0, 3)$, based on a likelihood $Y \mid \theta \sim \mathcal{N}(\theta,1)$.
Our assessment will be performed using the data-generating model
\begin{align*}
	P_\theta \; : \; Y \mid \theta \sim \begin{cases}
		\mathcal{N}(\theta, 1) & \text{w.p. } 1-\epsilon \\
		\mathcal{N}(5, 1) & \text{w.p. } \epsilon
	\end{cases} ,
\end{align*}
where $\epsilon \in [0,1]$ is a probability of obtaining a contaminated observation, so that for $\epsilon > 0$ the likelihood is misspecified and the Bayesian learning procedure is not strongly calibrated to $(\mu_0,P)$.
Fractional posteriors with exponent $t \in [0,1]$, as defined in \Cref{eg:power}, have been proposed as learning procedures that can offer robustness to misspecification of the likelihood e.g.\ in \cite{Grnwald2017}.
Our aim is to assess this claim within our testing framework.


Results of performing a \ac{ks} test of the strong calibration null hypothesis, using the identity test function $f(\theta) = \theta$, are displayed for a variety of values of $t$ and $\epsilon$ in \Cref{fig:contamination_strong_tests}.
Clearly the only circumstance in which any of the learning procedures is strongly calibrated is when $\epsilon = 0$ and the Bayesian procedure is used.
Otherwise, according to the test statistic in the left panel, fractional posteriors are marginally better calibrated when $\epsilon > 0$ than the Bayesian procedure, though regarding the $p$-values in the right panel one sees that the values of the statistic in these cases are still sufficiently sufficiently large to emphatically reject the strong calibration null hypothesis.

\begin{figure}
	\begin{subfigure}{0.49\textwidth}
	\includegraphics[width=\textwidth]{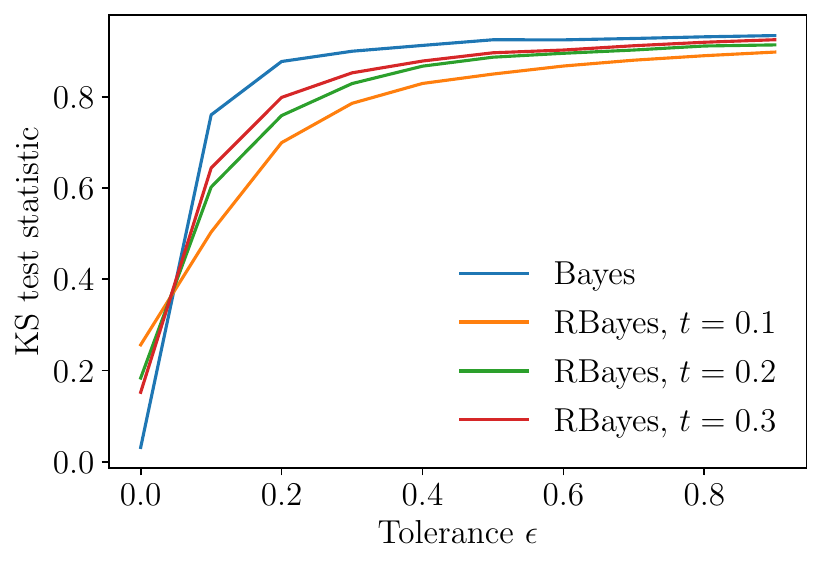}
	\end{subfigure}
	\begin{subfigure}{0.49\textwidth}
	\includegraphics[width=\textwidth]{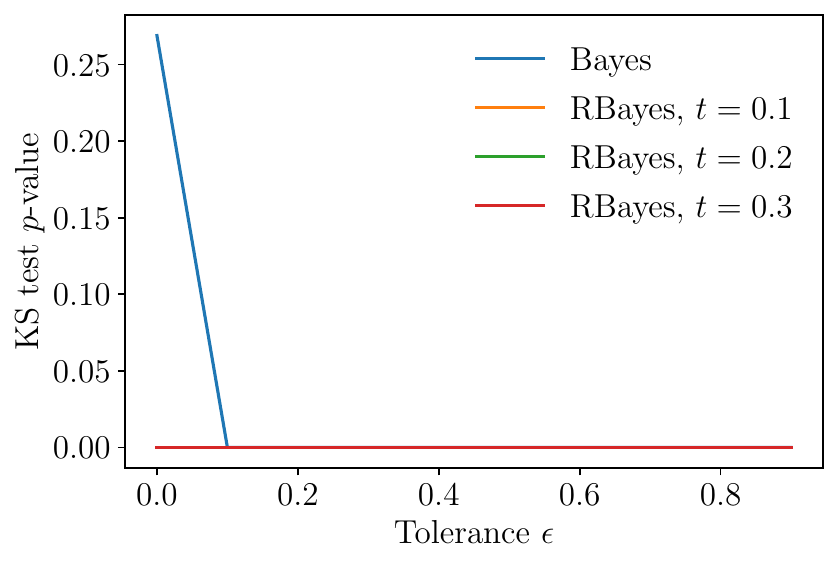}
	\end{subfigure}
	\caption{Robust calibration: Results of \ac{ks} tests for strong calibration, comparing standard Bayesian inference (``Bayes'') to the method of fractional posteriors (also called \textit{robust Bayes}; ``RBayes'') with exponents $t \in \{0.1,0.2,0.3\}$, in a setting where the likelihood is misspecified.
	Note that in the right panel the RBayes lines for $t=0.1, t=0.2$ and $t=0.3$ coincide.
	} \label{fig:contamination_strong_tests}
\end{figure}

Finally, we note that other senses of ``robust calibration'' could be considered, analogous to the various notions of ``robust estimation'' that have been studied \citep{Berger1994,huber2009robust}.
For example, one could consider a setting where true parameters $\theta$ are drawn from a distribution other than $\mu_0$ and assess the consequences, in terms of calibration, for a learning procedure that uses $\mu_0$ as the initial belief distribution.

\section{Discussion} \label{sec: discuss}

The desire that a parameter used to generate a dataset should appear plausible as a sample from the distributional output of a learning procedure, such as a Bayesian posterior, is foundational and, at least in an informal sense, widely understood and accepted.
Despite this, a precise and widely applicable notion of what it means for a learning procedure to be ``calibrated'' appears not to have been put forward.
Our aim in this paper was to propose such a definition, together with a framework for testing whether a learning procedure is calibrated.
In particular, we proposed a property called \textit{strong calibration} (\Cref{def: strong}), which provides an explicit sense in which output from the learning procedure can be considered to be meaningful.
A strictly weaker property, called \textit{weak calibration}, was also proposed (\Cref{def:calibrated}), which has the advantage of being more straightforward to test.
Several vignettes were provided to illustrate the generality and usefulness of the framework.

Our hope, in writing this manuscript, is to stimulate further critical discussion around calibration as a \textit{desideratum} for a learning procedure, and to bring together some of the disparate strands of literature where related concepts and domain-specific definitions have been developed (cf.\ \Cref{subsubsec: strong existing}).

\subsection{Further Work} \label{sec:future}

A particularly promising avenue for further research would be to develop \emph{measures} of miscalibration using the ideas proposed in this paper.
Generally speaking, when using approximate methods such as Laplace approximation (cf.\ \cref{subsec: Gauss}) or \ac{abc} (cf.\ \cref{subsec: ABC}), or generalised Bayesian methods (cf.\ \cref{sec: robust}), a user has purposefully departed from the Bayesian framework due to challenges such as its lack of computational tractability or the possibility that the model is misspecified.
In such settings a measure of miscalibration is likely to be of more use than a test for calibration, since exact calibration cannot be expected to hold.
A \emph{measure} of miscalibration might allow a user to select the ``most calibrated'' method from among multiple alternatives, or perhaps even incorporate calibration into a variational objective in a variational Bayesian framework \citep{Knoblauch2019}.

In the context of \cref{def: strong}, such a measure could be constructed by selecting some test function $f^* \in \mathcal{F}_\Theta$ and computing a statistical divergence between $F_{f^*_\# \mu(\mu_0, y)}(f^*(\theta))$ and $\mathcal{U}(0,1)$.
The former quantity is unlikely to be available in closed-form but could be estimated using Monte Carlo techniques.
Immediate challenges with this would concern selection of a suitable divergence and a suitable $f^*$. 
For the latter, one could perhaps instead consider selecting a subset $\mathcal{F}_\textrm{test} \subset \mathcal{F}_\theta$ over which a supremum can be taken tractably.
However, we leave this task for future work.

\section*{Acknowledgements}
\addcontentsline{toc}{section}{Acknowledgements}

JC was supported by Wave 1 of the UKRI Strategic Priorities Fund under the EPSRC Grant \href{https://gow.epsrc.ukri.org/NGBOViewGrant.aspx?GrantRef=EP/T001569/1}{EP/T001569/1}, particularly the ``Digital Twins for Complex Engineering Systems'' theme within that grant, and the Alan Turing Institute, UK.
MMG and CJO were supported by the Lloyd's Register Foundation programme on data-centric engineering at the Alan Turing Institute, UK.
TJS has been supported in part by the German Research Foundation (Deutsche Forschungsgemeinschaft) through project \href{https://gepris.dfg.de/gepris/projekt/415980428}{415980428} and the Excellence Cluster ``MATH+ The Berlin Mathematics Research Centre'' (EXC-2046/1, project \href{https://gepris.dfg.de/gepris/projekt/390685689}{390685689}).
The authors thank Dennis Prangle for feedback on an earlier version of the manuscript.

\bibliographystyle{abbrvnat}
\bibliography{bibliography}
\addcontentsline{toc}{section}{References}

\begin{appendices}

\crefalias{section}{appendix}

\section{Proof of Theoretical Results}

This appendix contains proofs for all novel results in the main text.
For $x \in \reals^d$, we let $(-\infty,x] \defeq (-\infty, x_1] \times \dots \times (-\infty, x_d]$ and we write $y \leq x$ whenever $y \in (-\infty,x]$, i.e.\ when $y_{i} \leq x_{i}$ for $i = 1, \dots, d$.

\subsection{Proof of \texorpdfstring{\Cref{lem:weak_control}}{Lemma \ref*{lem:weak_control}}} \label{sec:proof:weak_control}

Our proof of \Cref{lem:weak_control} makes use of the \textit{Kolmogorov distance}
\begin{align}
	\label{eq:Kolmogorov_distance}
	d_{\text{K}}(\mu,\nu) & \defeq \sup_{x \in \reals^d} d_{\text{K}}(x; \mu,\nu) , \\
	d_{\text{K}}(x; \mu,\nu) & \defeq \Absval{ \int 1_{(-\infty,x]} \, \rd \mu - \int 1_{(-\infty,x]} \, \rd \nu } ,
\end{align}
which is a metric on $\mathcal{P}(\reals^d)$ \citep[][Theorem~2.4]{shorack2000probability}. 

\vspace{5pt}

\begin{proof}[Proof of \Cref{lem:weak_control}]
	Suppose that $\mu \neq \nu$, so that it suffices to exhibit an element $f \in \mathcal{F}_\Theta$ for which $\int f \, \rd \mu \neq \int f \, \rd \nu$.
	From the metric property of $d_{\text{K}}$, there must exist $x^{\ast} \in \reals^d$ such that $\varepsilon \defeq d_{\text{K}}(x^{\ast}; \mu, \nu) > 0$.
	Now, for $c > 0$, consider the function
	\begin{align}
		f_{x^{\ast}}^{(c)} & \colon \reals^d \to (0,1), &
		f_{x^{\ast}}^{(c)}(x) & \defeq \prod_{i=1}^d \frac{1}{1 + e^{2c (x_i-x_i^*)}} , \label{eq:def_fc}
	\end{align}
	which satisfies $f_{x^{\ast}}^{(c)} \in \mathcal{F}_\Theta$.
	Since $f_{x^{\ast}}^{(c)}$ converges pointwise to $f_{x^{\ast}}^{(\infty)} \defeq 1_{(-\infty,x^{\ast}]}$ outside of a null set and $|f_{x^{\ast}}^{(c)}| \leq 1$, the dominated convergence theorem implies that $f_{x^{\ast}}^{(c)}$ is a consistent approximation of $f_{x^{\ast}}^{(\infty)}$ in the $c \to \infty$ limit in both $L^1(\mu)$ and $L^1(\nu)$.
	Therefore, there exists $c^{\ast} > 0$ such that $\|f_{x^{\ast}}^{(c^{\ast})} - f_{x^{\ast}}^{(\infty)}\|_{L^1(\mu)} < \varepsilon / 2$ and $\|f_{x^{\ast}}^{(c^{\ast})} - f_{x^{\ast}}^{(\infty)} \|_{L^1(\nu)} < \varepsilon / 2$.
	For this $f_{x^{\ast}}^{(c^{\ast})} \in \mathcal{F}_\Theta$ we have from the reverse triangle inequality that
	\begin{align*}
		\Absval{ \int f_{x^{\ast}}^{(c^{\ast})} \, \rd \mu - \int f_{x^{\ast}}^{(c^{\ast})} \, \rd \nu }
		& = \left| \left( \int f_{x^{\ast}}^{(c^{\ast})} \, \rd \mu - \int f_{x^{\ast}}^{(\infty)} \, \rd \mu \right) + \left( \int f_{x^{\ast}}^{(\infty)} \, \rd \mu - \int f_{x^{\ast}}^{(\infty)} \rd \nu \right) \right. \\
		& \qquad  \left. +  \left( \int f_{x^{\ast}}^{(\infty)} \, \rd \nu - \int f_{x^{\ast}}^{(c^{\ast})} \, \rd \nu \right) \right| \\
		& \geq \Biggl| \underbrace{ \left| \left( \int f_{x^{\ast}}^{(c^{\ast})} \, \rd \mu - \int f_{x^{\ast}}^{(\infty)} \, \rd \mu \right) + \left( \int f_{x^{\ast}}^{(\infty)} \, \rd \nu - \int f_{x^{\ast}}^{(c^{\ast})} \, \rd \nu \right) \right| }_{(\ast)} \\
		& \qquad - \underbrace{ \left| \int f_{x^{\ast}}^{(\infty)} \, \rd \mu - \int f_{x^{\ast}}^{(\infty)} \, \rd \nu \right| }_{=\varepsilon} \Biggr| .
	\end{align*}
	The triangle inequality implies that
	\[
		|(\ast)| \leq \norm{ f_{x^{\ast}}^{(c^{\ast})} - f_{x^{\ast}}^{(\infty)} }_{L^1(\mu)} + \norm{ f_{x^{\ast}}^{(\infty)} - f_{x^{\ast}}^{(c^{\ast})} }_{L^1(\nu)} < \varepsilon / 2 + \varepsilon / 2 = \varepsilon ,
	\]
	and so it follows that $\Absval{ \int f_{x^{\ast}}^{(c^{\ast})} \, \rd \mu - \int f_{x^{\ast}}^{(c^{\ast})} \rd \nu } \neq 0$.
	Thus we have exhibited an element $f_{x^{\ast}}^{(c^{\ast})} \in \mathcal{F}_\Theta$ for which $\int f_{x^{\ast}}^{(c^{\ast})} \, \rd \mu \neq \int f_{x^{\ast}}^{(c^{\ast})} \, \rd \nu$.
	This completes the proof.
\end{proof}

\subsection{Proof of \texorpdfstring{\Cref{lem:S_implies_W}}{Lemma \ref*{lem:S_implies_W}}} \label{sec:proof:S_implies_W}

First we derive a corollary of \Cref{lem:weak_control} that will be used to prove \Cref{lem:S_implies_W}:

\begin{corollary}
	\label{cor:leq_char}
	Let $\Theta = \reals^d$ for some $d \in \naturals$.
	Suppose that $\mu,\nu \in \mathcal{P}_r(\Theta)$ and that the independent random variables $\theta \sim \mu$, $\vartheta \sim \nu$ satisfy $\mathbb{P}(f(\theta) \leq f(\vartheta)) = 1/2$ for all $f \in \mathcal{F}_\Theta$.
	Then $\mu = \nu$.
\end{corollary}

\begin{proof}
	If $\mu \neq \nu$ then, as in the proof of \Cref{lem:weak_control}, we can identify $x^{\ast} \in \reals^d$ such that $d_{\text{K}}(x^{\ast}; \mu, \nu) > 0$.
	Since $\mu$ and $\nu$ are regular, the function $x \mapsto d_{\text{K}}(x; \mu, \nu)$ is continuous on $\reals^d$ and there exists an open neighbourhood $N(x^{\ast})$ of $x^{\ast}$ such that $d_{\text{K}}(x; \mu, \nu) > 0$ for all $x \in N(x^{\ast})$.

	Suppose, to arrive at a contradiction, that $\mathbb{P}(f(\theta) \leq f(\vartheta)) = 1/2$ for all $f \in \mathcal{F}_\Theta$.
	Then, for all $x \in N(x^{\ast})$, we can construct functions $f_x^{(c)} \in \mathcal{F}_\Theta$ as per \eqref{eq:def_fc}, for which it holds that
	\[
		\mathbb{P}(\vartheta \leq x) = \mathbb{P}(f_x^{(\infty)}(\theta) \leq f_x^{(\infty)}(\vartheta)) = \lim_{c \to \infty} \mathbb{P}(f_x^{(c)}(\theta) \leq f_x^{(c)}(\vartheta)) = \frac{1}{2} .
	\]
	But $\nu$ was assumed to be regular, meaning that $\nu$ has a positive Lebesgue \ac{pdf}, so that $\mathbb{P}(\vartheta \leq x) = 1/2$ cannot simultaneously hold for all $x \in N(x^{\ast})$.
	Indeed, since $N(x^{\ast})$ is open, there exists $x \in N(x^{\ast})$ such that $x_i^{\ast} < x_i$ for all $i = 1,\dots,d$.
	Then $\mathbb{P}(\vartheta \leq x) = \mathbb{P}(\vartheta \leq x^{\ast}) + \nu(S)$, where $S := (-\infty, x] \setminus (-\infty,x^{\ast}]$ is a measurable set with $\nu(S) > 0$.
	This contradiction completes the proof.
\end{proof}

\begin{proof}[Proof of \Cref{lem:S_implies_W}]
	Fix $\mu_0 \in B$.
	Let $\theta \sim \mu_0$, $y | \theta \sim P_\theta$ and $\vartheta | \theta, y \sim \mu(\mu_0,y)$.
	First we argue that the distribution $\nu \defeq \iint \mu(\mu_0,y) \, \wrt P_\theta (y) \, \wrt\mu_0(\theta)$ of the random variable $\vartheta$ is regular.
	Since $\mu$ is a regular learning procedure, $\mu(\mu_0,y)$ admits a \ac{pdf} $p_{\mu(\mu_0,y)}$ for each $y \in \reals^d$.
	Thus, $\nu$ admits the \ac{pdf}
	\[
		p_\nu(x) \defeq \int p_{\mu(\mu_0,y)}(x) \, \rd Q(y), \qquad Q \defeq \int P_\theta \, \rd \mu(\theta)
	\]
	and our task is to establish that this \ac{pdf} is positive on $\reals^d$.
	Fix $x \in \reals^d$.
	Now, since $p_{\mu(\mu_0,y)}(x) > 0$ for all $y \in \reals^d$, we have
	\[
		\reals^d = \bigcup_{n \in \naturals} S_n, \qquad S_n \defeq \left\{ y \in \reals^d \left| p_{\mu(\mu_0,y)}(x) > \frac{1}{n} \right. \right\} .
	\]
	Since $Q$ is a probability distribution on $\reals^d$, it follows that for some $n \in \naturals$, $Q(S_n) > 0$.
	Therefore
	\[
		p_\nu(x) = \int p_{\mu(\mu_0,y)}(x) \, \rd Q(y) > \frac{1}{n} Q(S_n) > 0
	\]
	and, since this argument holds for all $x \in \reals^d$, $p_\nu$ is a positive \ac{pdf} on $\reals^d$ and $\nu$ is regular.

	Next, since $\mu_0$ and the learning procedure $\mu$ are regular, and $\mu$ is strongly calibrated, for each $f \in \mathcal{F}_\Theta$,
	\begin{align}
		F_{f_\# \mu(\mu_0,y)}(f(\theta)) = \mathbb{P}(f(\vartheta) \leq f(\theta) | \theta, y) \sim \mathcal{U}(0,1)
	\end{align}
	and taking expectations of both sides yields
	\begin{align}
		\label{eq:half}
		\mathbb{P}(f(\vartheta) \leq f(\theta) ) = \frac{1}{2}.
	\end{align}
	Since both $\mu_0$ and $\nu$ are regular, it follows from \Cref{cor:leq_char} and \eqref{eq:half} that $\mu_0 = \nu$, and so $\vartheta$ has the marginal distribution $\mu_0$.
	Thus we have shown that the learning procedure $\mu$ is weakly calibrated to the belief distribution $\mu_0$ and the data-generating model $P$.
\end{proof}

\section{Probabilistic Numerical Methods for ODEs}
\label{app: PN}
This appendix contains full details of how the \acp{pnm} in \Cref{subsec: ODE calib} were implemented:

\vspace{5pt}
\begin{itemize} 
\item The code for \citet{chkrebtii16} was taken from $\mathtt{git.io/J33lL}$ and the step-size was set at $h=0.1$. The following settings were used: $\mathtt{nsolves}=100$, $\mathtt{N}=100$, $\mathtt{nevalpoints}=500$, $\mathtt{lambda} = 0.08$ and $\mathtt{alpha} = 1$. 
These values were manually selected, over the default values recommended in the code, since they led to improved calibration of the output.
Rigorous optimisation of these settings was not attempted.

\item The code for \citet{teymur18} was provided to us by the authors and is not yet publicly released. The method used is the 2-step (i.e. order 3) probabilistic Adams--Moulton method with step-size $h=0.5$ and overall scaling parameter $\alpha = 0.3$. 
These values were manually selected with the intention of improving calibration of the output, but rigorous optimisation of these settings was not attempted.
The stepwise perturbations are scaled using the global calibration procedure described in \cite{conrad16}.

\item The code for both \cite{schober18} and \cite{Tronarp18} derives from the comprehensive open-source Python package $\mathtt{probnum}$. 
On the advice of the authors of this package we implemented the adaptive routine $\mathtt{probnum/diffeq.probsolve\_ivp}$.
In this case the default values of tolerances were used. 
The only hyperparameter it is required to set is $\mathtt{algo\_order}$, which we set to $\mathtt{3}$. 
The setting $\mathtt{method=EK0}$ corresponds to \cite{schober18}, and $\mathtt{method=EK1}$ corresponds to \cite{Tronarp18}.

\item The code for \citet{teymur21} was provided to us by the authors and expected to be made public on full publication of that paper. 
This method is based on multi-fidelity simulation, so we take $h\in \{0.1,0.2,0.4\}$ and solve the {\scshape{ode}} using a 2-step (i.e. order 2) Adams--Bashforth method. All other hyperparameters are optimised automatically as part of the routine.
\end{itemize}

\end{appendices}
\end{document}